\documentclass[reqno]{amsart}

\input xy
\xyoption{all}
\usepackage{accents}
\usepackage{epsfig}
\usepackage{color}
\usepackage{amsthm}
\usepackage{amssymb}
\usepackage{amsmath}
\usepackage{amscd}
\usepackage{amsopn}
\usepackage{graphicx}
\usepackage{tikz-cd} 
\usepackage{xspace}

\usepackage{hhline}
\usepackage{easybmat}

\usepackage{url}
\usepackage{enumitem, hyperref}\hypersetup{colorlinks}

\usepackage{color} 

\definecolor{darkred}{rgb}{1,0,0} 
\definecolor{darkgreen}{rgb}{0,0.8,0}
\definecolor{darkblue}{rgb}{0,0,1}

\hypersetup{colorlinks,
linkcolor=darkblue,
filecolor=darkgreen,
urlcolor=darkred,
citecolor=darkgreen}

\makeatletter
\def\reflb#1#2{\begingroup
    #2%
    \def\@currentlabel{#2}%
    \phantomsection\label{#1}\endgroup
}
\makeatother

\numberwithin{equation}{section}
\newtheorem {Theorem}{Theorem}
\numberwithin{Theorem}{section}

\newtheorem {Proposition}[Theorem]{Proposition}
\newtheorem {Corollary}[Theorem]{Corollary}
\theoremstyle{definition}
\newtheorem{Definition}[Theorem]{Definition}
\theoremstyle{remark}
\newtheorem{Remark}[Theorem]{Remark}
\newtheorem{Example}[Theorem]{Example}

\def    \l    {\lambda}

\newcommand{\CE}{{\mathcal E}}

\newcommand{\CA}{{\mathcal A}}

\newcommand{\CM}{{\mathcal M}}

\newcommand{\Ham}{{\mathit{Ham}}}
\newcommand{\tHam}{{\widetilde{\mathit{Ham}}}}

\newcommand{\pt}{{\mathit pt}}

\newcommand{\tP}{\tilde{P}}

\newcommand{\tzeta}{\tilde{\zeta}}

\newcommand{\tpsi}{\tilde{\psi}}

\def    \nat    {{\natural}}
\def    \F      {{\mathbb F}}

\def    \R      {{\mathbb R}}

\def    \Z      {{\mathbb Z}}
\def    \N      {{\mathbb N}}
\def    \Q      {{\mathbb Q}}

\def    \T      {{\mathbb T}}
\def    \CP     {{\mathbb C}{\mathbb P}}

\def    \12    {{\frac{1}{2}}}

\def    \p      {\partial}
\def    \codim  {\operatorname{codim}}

\def    \Sp     {\operatorname{Sp}}
\def    \tSp     {\operatorname{\widetilde{Sp}}}
\def    \U     {\operatorname{U}}

\def    \HF     {\operatorname{HF}}

\def    \HQ     {\operatorname{HQ}}
\def    \GW     {\operatorname{GW}}

\def    \H     {\operatorname{H}}

\def    \CF     {\operatorname{CF}}

\def    \bx     {\bar{x}}
\def    \by     {\bar{y}}

\def    \vtheta     {\vec{\theta}}
\def    \vl    {\vec{\lambda}}

\def    \sgn  {\mathit{sgn}}

\def    \loopp  {\mathit{loop}}

\def    \MURS  {\operatorname{\mu_{\scriptscriptstyle{RS}}}}

\def  \hmu      {\operatorname{\hat{\mu}}}

\def \qq   {\mathrm{q}}
\def    \gr {\operatorname{gr}}

\def    \coker     {\operatorname{coker}}

\begin{document}


\setlength{\smallskipamount}{6pt}
\setlength{\medskipamount}{10pt}
\setlength{\bigskipamount}{16pt}





\title[Pseudo-Rotations and Holomorphic Curves]{Pseudo-Rotations and
  Holomorphic Curves}

\author[Erman \c C\. inel\. i]{Erman \c C\. inel\. i}
\author[Viktor Ginzburg]{Viktor L. Ginzburg}
\author[Ba\c sak G\"urel]{Ba\c sak Z. G\"urel}

\address{EC and VG: Department of Mathematics, UC Santa Cruz, Santa
  Cruz, CA 95064, USA} \email{scineli@ucsc.edu}
\email{ginzburg@ucsc.edu}

\address{BG: Department of Mathematics, University of Central Florida,
  Orlando, FL 32816, USA} \email{basak.gurel@ucf.edu}



\date{\today} 

\thanks{The work is partially supported by NSF CAREER award
  DMS-1454342, NSF grant DMS-1440140 through MSRI (BG) and by Simons
  Collaboration Grant 581382 (VG)}


\begin{abstract}
  We prove a variant of the Chance--McDuff conjecture for
  pseudo-rotations: under certain additional conditions, a closed
  symplectic manifold which admits a Hamiltonian pseudo-rotation must
  have deformed quantum product and, in particular, some non-zero
  Gromov--Witten invariants. The only assumptions on the manifold are
  that it is weakly monotone and that its minimal Chern number is
  greater than one. The conditions on the pseudo-rotation are
  expressed in terms of the linearized flow at one of the fixed points
  and hypothetically satisfied for most (but not all)
  pseudo-rotations.
\end{abstract}

\maketitle

\tableofcontents

\section{Introduction}
\label{sec:intro}

We show that a closed symplectic manifold, which has minimal Chern
number greater than one and admits a Hamiltonian pseudo-rotation
satisfying certain mild additional conditions, must have non-vanishing
Gromov--Witten invariants and, moreover, its quantum product is
deformed, i.e., different from the intersection product.

To put this result in perspective, recall that by the Conley
conjecture, for many symplectic manifolds every Hamiltonian
diffeomorphism has infinitely many periodic points. Obviously, the
conjecture requires some additional assumptions on the manifold: an
irrational rotation of $S^2$ about the $z$-axis has only two periodic
points: these are the fixed points -- the Poles. In a similar vein,
the conjecture fails for some other manifolds such as complex
projective spaces, Grassmannians and flag manifolds, symplectic toric
manifolds, and most of the coadjoint orbits of compact Lie groups. In
fact, the conjecture fails for all manifolds admitting a Hamiltonian
circle (or torus) action with isolated fixed points -- a generic
element of the circle or the torus gives rise to a Hamiltonian
diffeomorphism with finitely many periodic points.

On the conjecture side, these counterexamples are comparatively rare
and in a series of works easily spending three decades and contributed
by many, the Conley conjecture has been proved in many cases. The
state of the art result is that it holds for $M$ unless there exists
$A\in\pi_2(M)$ such that $\left<\omega,A\right>>0$ and
$\left<c_1(TM),A\right>>0$; see \cite{Ci,GG:Rev} and also
\cite{GG:survey} for further references and a thorough discussion. In
particular, the conjecture holds whenever $M$ is symplectically
aspherical or negative monotone or $\omega\mid_{\pi_2(M)}=0$.

Yet, these purely topological conditions leave aside a more subtle
question of symplectic topological criteria for the Conley conjecture
to hold. In that realm, the outstanding problem, referred to as the
Chance--McDuff conjecture, is that whenever the Conley conjecture
fails some Gromov--Witten invariants of $M$ are non-zero. It is
well-known that there is a strong connection between the symplectic
topology of $M$ (e.g., Gromov--Witten invariants or the quantum
product) and the dynamics (periodic orbits) of Hamiltonian
diffeomorphisms $\varphi$ of $M$. However, this connection is explored
and usually utilized only in one direction: from symplectic topology
to dynamics. The difficulty in proving the Chance--McDuff conjecture
lies in that it requires going in the opposite direction and this is a
much less understood problem. Till now the only work along these lines
was \cite{McD} where it is shown that a symplectic manifold admitting
a Hamiltonian circle action is uniruled, i.e., has a non-zero
Gromov--Witten invariant with one of the homology classes being the
point class.

In this context, pseudo-rotations are, roughly speaking, Hamiltonian
diffeomorphisms with a finite and minimal possible number of periodic
points.  (Actual definitions vary, but all of them reflect the same
idea; see \cite{GG:PR}.) In particular, pseudo-rotations are
counterexamples to the Conley conjecture and, in fact, they are the
only counterexamples known to date. (See \cite{Sh:HZ} for some
relevant recent results.)  Every known Hamiltonian diffeomorphism
$\varphi$ with finitely many periodic points is a pseudo-rotation in a
very strong sense: all periodic points of $\varphi$ are its fixed
points, they are elliptic, and all iterates $\varphi^k$ are
non-degenerate. This is the definition we adopt here.

Pseudo-rotations occupy a distinguished place in dynamical systems
theory far and mainly beyond the Hamiltonian setting. They can have
extremely interesting dynamics. For instance, there are examples of
ergodic Hamiltonian pseudo-rotations and even of pseudo-rotations with
finite number of ergodic measures. Such pseudo-rotations are obtained
by the so-called conjugation method which requires the manifold to
have a circle or torus action; see \cite{AK,FK,LeRS}. In fact, in all
known examples, a manifold which admits a pseudo-rotation also admits
a circle or torus action. This, combined with the results from
\cite{McD}, was the main motivation for the Chance--McDuff
conjecture. Recently it has been understood that symplectic
topological methods are well suited for studying Hamiltonian
pseudo-rotations; \cite{Br:Annals,Br,BH,GG:PR, GG:PRvR}.

Here we prove a variant of the Chance--McDuff conjecture for
pseudo-rotations. Namely, we show that, under certain additional
conditions, a manifold $M$ that admits a pseudo-rotation $\varphi$
must have deformed quantum product and, in particular, some
non-vanishing Gromov--Witten invariants. The
only assumptions on $M$ are that it is weakly monotone and that $N>1$,
where $N$ is the minimal Chern number. The conditions on $\varphi$ are
more involved and phrased in terms of the linearized flow at one of
its one-periodic orbits. One may expect these conditions to be met for
the majority (although certainly not all) pseudo-rotations. When this
paper was near completion we learned about a work by Egor Shelukhin,
\cite{Sh}, where he also established a variant of the Chance--McDuff
conjecture.

Our method uses only a minimal input from symplectic
topology. However, on the unexpected side, it relates combinatorics of
integer partitions to the pair-of-pants product and the regularity of
zero-energy pair-of-pants curves in Floer theory; see Sections
\ref{sec:trans} and \ref{sec:mainthm-pf}.  We use these curves to
capture non-vanishing Gromov--Witten invariants. Identifying the
quantum and Floer homology, we show by purely combinatorial means that
in many instances there are abundant zero-energy pair-of-pants curves
corresponding to long products in quantum homology. These long
products would vanish if the quantum product were undeformed. The
underlying idea can be best illustrated by the example of an
irrational rotation of~$S^2$.

\begin{Example}[Irrational Rotations of $S^2$]
  \label{ex:irr-rot}
  Let $\varphi$ be an irrational rotation of $S^2$ in an angle
  $\theta$, where $\pi<\theta<2\pi$. The fixed points of $\varphi$ are
  the North Pole $y$ and the South Pole $x$. The iterates $y^k$ and
  $x^k$ are also the only periodic points of $\varphi$. We equip these
  points with trivial cappings. Then $\mu(y)=1$ and $\mu(x)=-1$. Thus,
  when we identify the Floer complex $\CF_*(\varphi)$ with the quantum
  homology $\HQ_*\big(S^2\big)[-1]$, the North Pole $y$ represents the
  fundamental class $[S^2]$ and the South Pole $x$ represents
  $[\pt]$. On the other hand, $\mu\big(x^2\big)=-3$. Thus $[x^2]$
  represents the class $\qq [S^2]$, where $\qq$ is the generator of
  the Novikov ring. (With our conventions $|\qq|=-4$.) There exists
  exactly one pair-of-paints curve from $(x,x)$ to $x^2$ -- the
  constant curve. Assuming that this pair-of-pants curve is regular,
  which indeed is the case (see Corollary \ref{cor:trans} and
  \cite{Se:prod}), we have
  $$
  x*x=x^2+\ldots,
  $$
  where $*$ is the quantum product and the dots stand for capped
  periodic orbits with action strictly smaller than the action of
  $x^2$. (In fact, it is easy to see that no such orbits enter this
  identity.) In any event, no cancellations can happen on the
  right-hand side and we conclude that
  $$
  [\pt]*[\pt]=\qq[S^2]+\ldots\neq 0.
  $$
  On the other hand, if the quantum product were not deformed (i.e.,
  agreed with the intersection product) we would obviously have
  $[\pt]*[\pt]=0$. (Moreover, we see that
  $\GW_A\big([\pt],[\pt],[\pt]\big)\neq 0$, where $A$ is the
  ``positive'' generator of $\H_2(S^2;\Z)$.)
\end{Example}

This method, which shares some common elements with \cite{Se:prod},
readily lends itself to several generalizations to be explored
elsewhere. First of all, by using other, more sophisticated algebraic
structures one can certainly alter the requirements on the
pseudo-rotation or, perhaps, even eliminate these requirements
entirely. Secondly, under favorable circumstances, the method allows
one to obtain more specific information about the quantum homology
algebra of $M$ although the combinatorics of the problem quickly gets
rather involved.

The paper is organized in a somewhat counter-logical fashion. In
Section \ref{sec:results} we give necessary definitions and state main
results. Preliminary material from symplectic topology is discussed in
Section \ref{sec:prelim}. In Section \ref{sec:adm_part-qp} we
introduce extremal partitions -- the key combinatorial ingredient of
the proofs -- and reduce the main results of the paper to
combinatorial problems. Extremal partitions are studied in detail in
Section \ref{sec:adm_part-study}, where we prove the combinatorial
counterparts of the main theorems and thus complete their proofs.

\section{Main results}
\label{sec:results}
To detect the quantum product, our method requires imposing some
additional conditions on a pseudo-rotation $\varphi$. These
requirements are often, but not always, satisfied and are expressed in
terms of the linearized flow $\Phi=D\varphi^t|_{\bx}$ along a capped
one-periodic orbit $\bx$ of $\varphi^t$. In this section, we first
formulate these conditions and then state the main results of the
paper, deliberately opting to work with requirements which are easier
to state rather than more general.

\subsection{Definitions}
\label{sec:def}
We start by introducing several symplectic linear algebra invariants
associated with the linearized time one-map (or the flow) at a
one-periodic orbit $x$ of a Hamiltonian diffeomorphism $\varphi$. In
the discussion below, the reader should think that
$P=D\varphi\colon T_xM\to T_xM$ at a fixed point $x$ of $\varphi$ and
$\Phi$ is the linearized flow $D\varphi^t|_{\bx}$ along a capped
one-periodic orbit $\bx$ and that $N$ is the minimal Chern number of
$M$.

A word is also due on the nomenclature used in this section: on
Conditions \ref{cond:A}, \ref{cond:B1} and \ref{cond:B2}. The reason
for this labeling is that the role of Condition \ref{cond:A} is
distinctly different from that of Conditions \ref{cond:B1} or
\ref{cond:B2}. Condition \ref{cond:A} is used to detect long
non-vanishing products in the quantum homology, while Conditions
\ref{cond:B1} and \ref{cond:B2} ensure that these products are not
essentially the products of the fundamental class with itself.

\subsubsection{Base group}
\label{sec:base_grp}
Consider an elliptic and non-degenerate symplectic transformation
$P\in \Sp(2n)$ and let $\tP\in\Sp(2n)$ be isospectral to $P$ and
semi-simple. In other words, we require that all eigenvalues of $\tP$
are unit, $1$ is not an eigenvalue, and there exists a family of
non-degenerate tranformations $P_t\in\Sp(2n)$ connecting $P_0=P$ and
$P_1=\tP$ such that all $P_t$ have the same spectrum, and $\tP$ is
diagonalizable, i.e., $\R^{2n}$ splits into a sum of $n$ invariant
symplectic subspaces. (Then each of these subspaces is a plane and on
it $\tP$ is conjugate to a rotation.) It is easy to see that such a
transformation $\tP$ exists and, in fact, can be taken arbitrarily
close to $P$; see, e.g., \cite[Lemma 5.1]{Gi:CC}. Of course, $\tP$ is
not unique.

Since $\tP$ is elliptic and semi-simple, it is symplectically
conjugate to a unitary transformation and, as a consequence, the
closure of the sequence $\{\tP^k\mid k\in\N\}$ is a compact abelian
subgroup of $\Sp(2n)$, which we denote by $\Gamma$ (or $\Gamma(P)$ or
$\Gamma(x)$ when $P=D\varphi|_x$), and call the \emph{base group}. By
construction, $\Gamma$ is monothetic and thus $\Gamma/\Gamma_0$ is
finite cyclic, where $\Gamma_0$ is the connected component of the
identity in $\Gamma$. Clearly, up to symplectic conjugation, $\Gamma$
is independent of the choice of $\tP$.

Alternatively, $\Gamma$ can be described as follows. Since, $P$ is
elliptic all eigenvalues of $P$ lie on the unit circle. Let
$$
\vtheta:=(\theta_1,\ldots,\theta_n) \in \T^n=S^1_1 \times \ldots
\times S^1_n
$$
be the collection of Krein--positive eigenvalues of $P$, ordered in an
arbitrary way; see, e.g., \cite[Sect.\ 1.3]{Ab} or \cite{Lo,SZ}. Then
$\Gamma$ is naturally isomorphic to the subgroup of the torus $\T^n$
generated by $\vtheta$, i.e., to the closure of the sequence
$\{k\vtheta\mid k\in \N\}$ in $\T^n$. For a suitable choice of a
complex structure on $\R^{2n}$ we can think of $\T^n$ as the maximal
torus in $\U(n)$ containing $\tP=\vtheta$. Note that
$\dim \Gamma\geq 1$ when $P$ is strongly non-degenerate, i.e., all
iterates $P^k$, $k\in\N$, are non-degenerate. (The converse is not
true.) The key point here is that the ``index theory'' for $\tP$ is
the same as for $P$, but the group generated by $P$ in $\Sp(2n)$ is
not compact, unless $P$ is semi-simple, and is much harder work with.

\begin{Example}
  \label{ex:base_grp1}
  Assume that $\varphi$ is a true rotation (i.e., $\varphi$ generates
  a compact subgroup $G$ of $\Ham(M)$, see Example \ref{ex:true-rot})
  or that it is obtained from such a rotation by the conjugation
  method. Let $P=D\varphi|_{x}$, where $x$ is a fixed point of
  $\varphi$. Then $P$ is automatically semi-simple and, for a true
  rotation, $\Gamma$ is the image of $G$ in $\Sp(T_xM)$ under the
  natural representation of $G$ on $T_xM$.
\end{Example}

\begin{Definition}[Condition A]
  \label{def:cond_A}
  For a fixed $r\in \N$, the transformation $P$ (or the subgroup
  $\Gamma=\Gamma(P)$ or the orbit $x$) satisfies \emph{Condition
    \reflb{cond:A}{A}} if there exist $r$ points
  $\vtheta_1,\ldots,\vtheta_r$ in $\Gamma$ such that
  \begin{equation}
    \label{eq:cond_A}
    \sum_{i=1}^r \lambda_{ij}<1 \textrm{ for all } j=1,\ldots, n,
\end{equation}
where we set
  $$
  \vtheta_i= \big(e^{2\pi\sqrt{-1} \lambda_{i1}}, \ldots,
  e^{2\pi\sqrt{-1}\lambda_{in}}\big) \textrm{ with } 0<\lambda_{ij}<1.
  $$
\end{Definition}
To see what this means geometrically, let us identify $\T^n$ with the
product $I_1\times \ldots \times I_n$ of $n$ intervals each of which
is $[0,\,1]$. Then \eqref{eq:cond_A} determines the standard open
simplex $\Delta$ in the cube $I_j^r$ and Condition \ref{cond:A} is
equivalent to that $\Gamma^r$ intersects the region in $(\T^n)^r$
obtained from the product of $r$ copies of $\Delta$ by rearranging the
coordinates. It is clear that Condition \ref{cond:A} is independent of
the choice of $\tP$ or the ordering of the eigenvalues of $P$.

\begin{Example}[Toric $\Phi$]
  \label{ex:base_grp2}
  Assume that $\Phi$ is \emph{toric}, i.e., by definition
  $\dim \Gamma=n$ or equivalently $\Gamma=\T^n$; see Section
  \ref{sec:toric}. Then Condition \ref{cond:A} is automatically
  satisfied and $P$ is elliptic, semi-simple and strongly
  non-degenerate. More generally, Condition \ref{cond:A} is met when
  that $\Gamma$ contains a one-parameter subgroup of the form
  $t \mapsto (a_1t, \ldots, a_nt)$, $t\in\R$, with $a_i>0$ for
  all~$i$; see Example \ref{ex:ccwise}.
\end{Example}

We will see later that Condition \ref{cond:A} is in some sense
satisfied for ``most'' of the elliptic transformations $P$.

Next, denote by $\mu_\Gamma\in \H^1(\Gamma;\Z)$ the restriction of the
Maslov class to $\Gamma$. In other words, consider the codimension-one
cocycle in $\Gamma$ which is the sum of $n$ cocycles obtained by
setting the $i$th coordinate $\theta_i\in S^1_i$ in $\T^n$ equal to 1
and co-oriented by the counterclockwise orientation of $S^1_i$. (We
are assuming here that $\Gamma$ is not contained in any of the subtori
$\theta_i=1$.) Then $\mu_\Gamma$ is the cohomology class of the this
cocycle. Note that the mean index $\hmu(\gamma)$ of a loop $\gamma$ in
$\Gamma$ is $2\mu_\Gamma(\gamma)$.

\begin{Definition}[Condition B1]
  \label{def:cond_B1}
  For a fixed $N\in \N$, the transformation $P$ (or the subgroup
  $\Gamma=\Gamma(P)$ or the orbit $x$) satisfies \emph{Condition
    \reflb{cond:B1}{B1}} if $\mu_\Gamma$ is not divisible by $N$,
  i.e., there exists a loop $\gamma$ in $\Gamma$ such that
  $N\not|\,\mu_\Gamma(\gamma)$.
\end{Definition}

\begin{Example}
  \label{ex:base_grp3}
  Assume that $P$ is toric, i.e., $\Gamma=\T^n$. Then Condition
  \ref{cond:B1} is automatically satisfied when $N>1$. On the other
  hand, let $\Gamma$ be the circle $t\mapsto (a_1t,\ldots,a_nt)$,
  $t\in S^1$, where $a_i\in \Z$ are non-zero and relatively
  prime. Then $\mu_\Gamma=a_1+\ldots+a_n$ in
  $\H^1(\Gamma;\Z)\cong\Z$. Thus Condition \ref{cond:B1} is satisfied
  if and only if $N\not|\, (a_1+\ldots+a_n)$. Finally, note that this
  condition is never met when $N=1$.
\end{Example}

We say that a path $\Phi \colon [0,\,1] \to \Sp(2n)$ satisfies
Conditions \ref{cond:A} and \ref{cond:B1} if the end-point $\Phi(1)$
satisfies these conditions.  We will elaborate on Conditions
\ref{cond:A} and \ref{cond:B1} in Sections \ref{sec:cases} and
\ref{sec:adm_part-prop}.

\subsubsection{Loop contribution}
\label{sec:loop}
Conditions \ref{cond:A} and \ref{cond:B1} are expressed entirely in
terms of the linear map $P$ or the group $\Gamma$. However, in our
case more information is available -- this is the linearized flow
along $x$ -- and in this section we utilize it.

Consider a strongly non-degenerate path
$\Phi\colon [0,\,1]\to\Sp(2n)$, which we view as an element of
$\tSp(2n)$, with end-point $P=\Phi(1)$. When $P$ is semi-simple we can
decompose $\Phi$ as the concatenation (or product) of a loop $\phi$
and a direct sum of $n$ ``short rotations''
$t\mapsto \exp\big(\pi\sqrt{-1}\lambda t\big)$, where $t\in [0,\,1)$
and $|\lambda|<1$; cf.\ \cite[Sect.\ 4]{GG:PRvR}. When $P$ is not
semi-simple we need to add an isospectral path $P_t$ to this
decomposition as in the previous section. In either case, as is easy
to see, the free homotopy class of the loop $\phi$ is uniquely
determined by $\Phi$. Equivalently, the mean index $\hmu(\phi)$ is
well defined. Set $\loopp(\Phi):=\hmu(\phi)$ and call $\loopp(\Phi)$
the \emph{loop part} of $\Phi$. (Note that $\loopp(\Phi)$ is
necessarily even and equal twice the Maslov class of $\phi$.)

\begin{Definition}[Condition B2]
  \label{def:cond_B2}
  For a fixed $N\in \N$, the path $\Phi$ (or a capped orbit $\bx$)
  satisfies \emph{Condition \reflb{cond:B2}{B2}} if $\Gamma$ is
  connected and there exists a convex neighborhood $V$ of $0\in \T^n$
  whose intersection with $\Gamma$ is connected and an iterate
  $\Phi^k(1)\in V$ such that $2N\not|\,\loopp\big(\Phi^k\big)$.
\end{Definition}

Roughly speaking, one should expect $N-1$ out of $N$ randomly taken
paths $\Phi$ to satisfy this condition. On the other hand, Condition
\ref{cond:B2} (just as Condition \ref{cond:B1}) is never satisfied
when $N=1$.

\subsection{Detecting the quantum product}
\label{sec:main-thm}
Let $(M^{2n},\omega)$ be a closed weakly monotone symplectic manifold
with minimal Chern number $N$. Fix a ground ring $\F$, suppressed in
the notation; e.g., $\F=\Z$ or $\Z_2$ or $\Q$. For our purposes it is
convenient to adopt the following definition; cf.\ \cite{GG:PR}.

\begin{Definition}[Pseudo-rotations]
  \label{def:PR}
  A Hamiltonian diffeomorphism $\varphi\colon M\to M$ is called a
  \emph{pseudo-rotation} (over $\F$) if $\varphi$ is strongly
  non-degenerate, and the differential in the Floer complex of
  $\varphi^k$ over $\F$ vanishes for all $k\in\N$.
\end{Definition}

The differential in the Floer complex depends on the almost complex
structure, but it is easy to see that its vanishing is a well-defined
condition. Note also that for a pseudo-rotation all periodic orbits
are automatically one-periodic and that an iterate of a
pseudo-rotation is again a pseudo-rotation. Definition \ref{def:PR} is
slightly different from the one in \cite{GG:PR} although it captures
the same phenomenon. We refer the reader to that paper for a detailed
discussion of various definitions of a pseudo-rotation. Finally, note
that in some of our results the non-degeneracy requirement can be
somewhat relaxed, but not entirely omitted.

\begin{Example}
  Assume that $\varphi$ is strongly non-degenerate and all its
  periodic orbits are elliptic. Then $\varphi$ is a
  pseudo-rotation. All known to date pseudo-rotations are of this
  type.
\end{Example}

\begin{Example}[True rotations]
  \label{ex:true-rot}
  Assume that $\varphi$ is a \emph{true rotation}, i.e., by definition
  $\varphi$ generates a compact (but not finite) subgroup $G$ of
  $\Ham(M)$. Then $G$ is necessarily a compact Lie group by \cite{RS},
  and hence its connected component $G_0$ of the identity is a
  torus. It is then a standard fact that $\varphi$ is strongly
  non-degenerate if and only if its periodic points are isolated and
  if and only if it has finitely many periodic orbits; see
  \cite{GGK}. Furthermore, the resulting $G_0$-action on $M$ is
  Hamiltonian. (This is ultimately a consequence of some deep results,
  starting with \cite{Ba} characterizing Hamiltonian diffeomorphisms
  as symplectomorphisms with zero flux and then the flux conjecture
  proved in \cite{On}; see also \cite{LMP1, LMP2}.)  In particular,
  strongly non-degenerate true rotations are among
  pseudo-rotations. All other known examples of pseudo-rotations are
  obtained from such true rotations by the conjugation method,
  \cite{AK, FK, LeRS}.
\end{Example}

Recall that the established cases of the Conley conjecture discussed
in the introduction limit the class of manifolds that can possibly
admit pseudo-rotations or more generally Hamiltonian diffeomorphisms
with finitely many periodic orbits, \cite{Ci, GG:Rev}. (Namely, when
$M$ admits a pseudo-rotation there exists $A\in\pi_2(M)$ such that
$\left<\omega,A\right>>0$ and $\left<c_1(TM),A\right>>0$. In
particular, $\omega|_{\pi_2(M)}\neq 0$ and
$c_1(TM)|_{\pi_2(M)}\neq 0$.) The following simple result, specific to
pseudo-rotations, further narrows down the class of such manifolds.

\begin{Proposition}
  \label{prop:PR-existence}
Assume that $M^{2n}$ admits a pseudo-rotation. Then $N\leq
2n$, where $N$ is the minimal Chern number of $M$.
\end{Proposition}

A similar result has been recently proved in \cite{Sh} under slightly
less restrictive conditions and by a different method. The upper bound
from Proposition \ref{prop:PR-existence} is extremely unlikely to be
sharp: $N\leq n+1$ for all known manifolds admitting
pseudo-rotations. We defer the proof of the proposition to Section
\ref{sec:floer}.

Denote by $\HQ_*(M)$ the (small) quantum homology of $M$, by $*$ the
quantum product, and by $|\alpha|$ the degree of an element
$\alpha\in \HQ_*(M)$. Recall that the quantum product is said to be
\emph{deformed} if it is not equal to the intersection product.

The main result of the paper is the following.

\begin{Theorem}
  \label{thm:main}
  Assume that $M^{2n}$ admits a pseudo-rotation $\varphi$ with an
  elliptic fixed point $x$ which, for some $r\in \N$, satisfies
  Condition \ref{cond:A} and also Condition \ref{cond:B1} or, for some
  capping, Condition \ref{cond:B2}. Then there exist $r$ elements
  $\alpha_1,\ldots,\alpha_r$ in $\HQ_*(M)$ of even degree such that
\begin{equation}
  \label{eq:prod}
  \alpha_1*\ldots*\alpha_r\neq 0
\end{equation}
and
\begin{equation}
  \label{eq:degs}
  |\alpha_i|\not\equiv 2n\mod 2N \textrm{ for all } i=1,\ldots, r.
\end{equation}

\end{Theorem}

This theorem is proved in Section \ref{sec:mainthm-pf}. The key
ingredient of the argument is a combinatorial result of independent
interest (Theorem \ref{thm:adm_part}) concerning certain long products
in $\tSp(2n)$ maximizing the defect of the Conley--Zehnder- or
Maslov-type quasimorphism. Here, we have tacitly assumed that $N$, for
which Condition \ref{cond:B1} or \ref{cond:B2} is satisfied, is the
minimal Chern number of $M$. However, the theorem holds for any $N$
meeting this requirement, which is actually a stronger result, and
\eqref{eq:degs} gets easier to satisfy as $N$ grows. 

\begin{Corollary}
  \label{cor:main}
  Assume that the conditions of Theorem \ref{thm:main} are satisfied
  with $N$ being the minimal Chern number of $M$ and $r\geq n+1$. Then
  the quantum product is deformed and, in particular, some
  Gromov--Witten invariants of $M$ are non-zero.
\end{Corollary}

Note that here we could have as well required that $N\geq 2$; for
Conditions \ref{cond:B1} and \ref{cond:B2} are never satisfied when
$N=1$.  We will see from the results in Section \ref{sec:cases} that,
while involved, the conditions of Theorem \ref{thm:main} and Corollary
\ref{cor:main} are satisfied in many cases. In fact, one can expect
them to be met for a majority of pseudo-rotations unless $N=1$. There
are, however, exceptions: when a pseudo-rotation is ``too symmetric''
the conditions might not be met for any fixed point; see Remark
\ref{rmk:diag-pr} below.

\begin{proof}[Proof of Corollary \ref{cor:main}] Write
  $\alpha_i=\sum_j f_j\alpha_{ij}$, where $f_j$ is an element of the
  Novikov ring of degree $j$ and
  $\alpha_{ij}\in \H_{\mathit{even}}(M)$. With our conventions $|f_j|$
  is divisible by $2N$; see Section \ref{sec:prelim}. By
  \eqref{eq:degs}, $|\alpha_{ij}|<2n$. Thus, if $*$ is equal to the
  cup product, we necessarily have
  $$
  \alpha_1*\ldots*\alpha_r=0
  $$
  when $r\geq n+1$; for $\alpha_{ij}$ is not proportional to the
  fundamental class.
\end{proof}

\begin{Remark}
  \label{rmk:diag-pr}
  While the conditions of Corollary \ref{cor:main} are probably met in
  most cases unless $N=1$, there are some exceptions. For instance,
  let $R_\theta\colon S^2\to S^2$ be the rotation in
  $\theta\not\in 2\pi\Q$. Then the diagonal map
  $\varphi=(R_\theta, R_\theta) \colon S^2 \times S^2 \to S^2 \times
  S^2$ does not have a periodic orbit $x$ meeting the requirements of
  the corollary or of Theorem \ref{thm:pr-dim4} below, which is
  sharper. Thus, for this pseudo-rotation, our method in its present
  form does not detect the deformed quantum product. (However, one can
  show that even in this case there are non-trivial Gromov--Witten
  invariants coming from zero energy pair-of-pants curves.)
\end{Remark}

\begin{Remark}[True rotations]
  \label{rmk:true_rot}
  Theorem \ref{thm:main} and Corollary \ref{cor:main} are not obvious
  even for true rotations. However, then $M$ admits a non-degenerate
  Hamiltonian circle action and the results from \cite{McD} guarantee
  non-vanishing of certain Gromov--Witten invariants. (In general, the
  argument in \cite{McD} does not require non-degeneracy and seems to
  take no advantage of it.) A direct comparison of the results from
  \cite{McD} and this paper is not straightforward. One could expect
  that for non-degenerate true rotations the results from \cite{McD}
  would be stronger than, say, Theorem \ref{thm:main}, but
  surprisingly this does not seem to be the case. It appears that the
  two methods in general detect different Gromov--Witten invariants.
\end{Remark}

\subsection{Particular cases and refinements}
\label{sec:cases}
In this section we discuss some particular cases and refinements of
Theorem \ref{thm:main} and Corollary \ref{cor:main}.

\subsubsection{Toric pseudo-rotations}
\label{sec:toric}
One case when the conditions of Theorem \ref{thm:main} are
automatically satisfied is when a pseudo-rotation behaves as a generic
element of the Hamiltonian $\T^n$-action on a toric symplectic
$2n$-dimensional manifold.

\begin{Definition}[Toric Pseudo-rotations]
  \label{def:toric}
  A pseudo-rotation $\varphi$ of a closed symplectic manifold $M^{2n}$
  is said to be \emph{toric} if it has a fixed point $x$ with
  $\dim \Gamma(x)=n$.
\end{Definition}

Note that in this case $x$ is necessarily strongly non-degenerate.

\begin{Corollary}
  \label{cor:toric}
  Assume that $M$ admits a toric pseudo-rotation and $N>1$. Then the
  quantum product is deformed and, in particular, some Gromov--Witten
  invariants of $M$ are non-zero.
\end{Corollary}

This corollary is an immediate consequence of Examples
\ref{ex:base_grp2} and \ref{ex:base_grp3} showing that Conditions
\ref{cond:A} and \ref{cond:B1} are automatically satisfied, and, of
course, of Theorem \ref{thm:main}. Moreover, then the theorem can be
refined as follows:

\begin{Theorem}
  \label{thm:toric}
  Assume that $M^{2n}$ admits a toric pseudo-rotation. Then, for every
  $r\geq 1$, there exists $\alpha \in \HQ_{2n-2}(M)$ such that
  $\alpha^r\neq 0$.
\end{Theorem}

The proof of this result is quite similar to the proof of Theorem
\ref{thm:main}; see Section~\ref{sec:mainthm-pf}.

\begin{Remark}
  Although the condition that a pseudo-rotation is toric appears
  generic -- and it is indeed generic in $P=D\varphi_x$ -- it is
  probably quite restrictive. Assume, for instance, that a toric
  pseudo-rotation is a true rotation; see Example
  \ref{ex:true-rot}. Then, as is easy to see from Example
  \ref{ex:base_grp1}, $M^{2n}$ is necessarily a toric symplectic
  manifold and $\varphi$ (or some iterate of it) is a topological
  generator of a Hamiltonian $\T^n$-action. One can expect only very
  few manifolds to have toric pseudo-rotations even among manifolds
  admitting pseudo-rotations, although this expectation is based more
  on the lack of knowledge and examples than on serious evidence. Note
  however that the proof of Theorem \ref{thm:toric} provides, at least
  in principle, a way to obtain detailed information about the quantum
  product structure for $M$ and it would be interesting to compare it
  with the quantum product for toric manifolds.
\end{Remark}

\subsubsection{Pseudo-rotations in dimension four}
\label{sec:pr-dim4}
When $\dim M=4$, i.e., $n=2$, which we assume throughout this section,
Theorem \ref{thm:main} and Corollary \ref{cor:main} can be further
refined. Here we focus on detecting the quantum product, although with
some more work the method can be used to get specific information
about the structure of the quantum product under minor assumptions on
the base group.

When $n=2$, the dimension of the base group $\Gamma$ is either 0 or 1
or 2. If $\dim\Gamma=0$, some iteration of $P$ is necessarily
degenerate. When $\dim \Gamma=2$, Theorem \ref{thm:toric}
applies. Hence, here we concentrate on the case where $\dim\Gamma
=1$. Then the connected component $\Gamma_0$ of the identity in
$\Gamma$ is given by the equation
$$
s_1\theta_1+s_2\theta_2=0
$$
on $\T^2$ with angular coordinates $(\theta_1,\theta_2)$, where $s_1$
and $s_2$ are relatively prime integers. When $P$ is non-degenerate,
$s_1\neq 0$ and $s_2\neq 0$. We refer to the ratio $s=-s_1/s_2$ as the
\emph{slope} of~$\Gamma$.

By Proposition \ref{prop:PR-existence}, to admit a pseudo-rotation the
manifold $M$ must have the minimal Chern number $N\leq 4$. When $N=1$
our method does not detect the quantum product. The case of $N=4$ is
extremely hypothetical and not considered here. The following
theorem gives a rather precise criterion for $N=2$ and $3$.

\begin{Theorem}
  \label{thm:pr-dim4}
  Assume that $\dim M=4$ and $M$ admits a pseudo-rotation with an
  elliptic fixed point $x$ such that $\dim\Gamma(x)=1$. Assume
  furthermore that $N=2$ and $s\neq \pm 1,\, 3,\, 1/3,\, -2,\, -1/2$
  or $N=3$ and $s\neq -1, \, \pm 2, \, \pm 1/2$. Then the quantum
  product is deformed.
\end{Theorem}

Thus, in dimension four (with $N\geq 2$), the quantum product is
deformed whenever $\Gamma_0$ is not one of these six undesirable
subgroups. Here, of course, the case of $N=2$ is by far most
interesting; the only example of a 4-manifold with $N=3$ which admits
a pseudo-rotation known to us is $\CP^2$. The proof of the theorem is
based on detecting the product of length $r=3$ with \eqref{eq:degs}
satisfied; see Section \ref{sec:mainthm-pf}. In a similar vein, to
detect a product of any length $r$ it would be sufficient to rule out
only a finite number of subgroups. Finally, note that the requirements
of Theorem \ref{thm:pr-dim4} are strictly weaker than Conditions
\ref{cond:A} and \ref{cond:B1}. For instance, Condition \ref{cond:B1}
holds if and only if $s_1+s_2$ is odd. (However, Condition
\ref{cond:A} in dimension four is met for every $r$ by all but a
finite number of subgroups $\Gamma$ of positive dimension; see Remark
\ref{rmk:Cond-A-2D}.)  One can think of Theorem \ref{thm:pr-dim4} as
an additional proof of concept result: ultimately the method should
enable one to treat many more cases than covered by Theorem
\ref{thm:main}, although the combinatorics of the proof might get
rather involved.

\section{Preliminaries}
\label{sec:prelim}
In this section we set the conventions and notation used in the paper
and briefly recall several definitions and facts from symplectic
topology relevant for the proofs. We also prove Proposition
\ref{prop:PR-existence} and the regularity results for zero-energy
pair-of-pants curves.

\subsection{Conventions and notation}
\label{sec:conventions}
Throughout the paper $(M^{2n},\omega)$ is a closed symplectic
manifold, which, to avoid foundational issues, we will always assume
to be weakly monotone in the sense of \cite{HS}. The minimal Chern
number, i.e., the positive generator of the group
$\left< c_1(TM), \pi_2(M) \right> \subset \Z$, is denoted by
$N$. (When this group is zero, $N=\infty$.)

A \emph{Hamiltonian diffeomorphism} is the time-one map
$\varphi=\varphi_H$ of the time-dependent flow $\varphi_H^t$ of a
$1$-periodic in time Hamiltonian $H\colon S^1\times M\to\R$, where
$S^1=\R/\Z$.  The Hamiltonian vector field $X_H$ of $H$ is defined by
$i_{X_H}\omega=-dH$. Such time-one maps form the group
$\Ham(M,\omega)$ of Hamiltonian diffeomorphisms of $M$. In what
follows, it will be convenient to view Hamiltonian diffeomorphisms as
elements of the universal covering $\tHam(M,\omega)$.

Let $x\colon S^1\to M$ be a contractible loop. A \emph{capping} of $x$
is an equivalence class of maps $A\colon D^2\to M$ such that
$A\mid_{S^1}=x$. Two cappings $A$ and $A'$ of $x$ are equivalent if
the integrals of $\omega$ and $c_1(TM)$ over the sphere obtained by
attaching $A$ to $A'$ are equal to zero. A capped closed curve
$\bar{x}$ is, by definition, a closed curve $x$ equipped with an
equivalence class of cappings, and the presence of capping is always
indicated by a bar.

The action of a Hamiltonian $H$ on a capped closed curve
$\bar{x}=(x,A)$ is
$$
\CA_H(\bar{x})=-\int_A\omega+\int_{S^1} H_t(x(t))\,dt.
$$
The space of capped closed curves is a covering space of the space of
contractible loops, and the critical points of $\CA_H$ on this space
are exactly the capped one-periodic orbits of $X_H$.

The $k$-periodic \emph{points} of $\varphi_H$ are in one-to-one
correspondence with the $k$-periodic \emph{orbits} of $H$, i.e., of
the time-dependent flow $\varphi_H^t$. Recall also that a $k$-periodic
orbit of $H$ is called \emph{simple} or \emph{prime} if it is not
iterated. Clearly, the action functional is homogeneous with respect
to iteration: $\CA_{H^{\nat k}}\big(\bx^k\big)=k\CA_H(\bx)$, where
$\bx^k$ is the $k$th iteration of the capped orbit $\bx$. (The capping
of $\bx^k$ is obtained from the capping of $\bx$ by taking its
$k$-fold cover branched at the origin.)

A $k$-periodic orbit $x$ of $H$ is said to be \emph{non-degenerate} if
the linearized return map
$d\varphi_H^k \colon T_{x(0)}M \to T_{x(0)}M$ has no eigenvalues equal
to one. We call $x$ \emph{strongly non-degenerate} if all iterates
$x^k$ are non-degenerate. A Hamiltonian $H$ is non-degenerate if all
its one-periodic orbits are non-degenerate and $H$ is strongly
non-degenerate if all periodic orbits of $H$ (of all periods) are
non-degenerate.

Let $\bar{x}$ be a non-degenerate capped periodic orbit.  The
\emph{Conley--Zehnder index} $\mu(\bar{x})\in\Z$ is defined, up to a
sign, as in \cite{Sa,SZ}. In this paper, we normalize $\mu$ so that
$\mu(\bar{x})=n$ when $x$ is a non-degenerate maximum (with trivial
capping) of an autonomous Hamiltonian with small Hessian. The
\emph{mean index} $\hmu(\bx)\in\R$ measures, roughly speaking, the
total angle swept by certain (Krein--positive) unit eigenvalues of the
linearized flow $d\varphi^t_H|_{\bx}$ with respect to the
trivialization associated with the capping; see \cite{Lo,SZ}. The mean
index is defined even when $x$ is degenerate and depends continuously
on $H$ and $\bx$ in the obvious sense.  Furthermore,
$$
\big|\hmu(\bx)-\mu(\bx)\big|\leq n .
$$
The mean index is homogeneous with respect to iteration:
$\hmu\big(\bx^k\big)=k\hmu(\bx)$. For an uncapped orbit $x$, the mean
index $\hmu(x)$ is well defined as an element of
$S^1_{2N}:=\R/2N\Z$. Likewise, when $x$ is non-degenerate, the
Conley--Zehnder index $\mu(x)$ is well defined as an element of
$\Z/2N\Z$.

\subsection{Floer homology and the pair-of-pants product}
\subsubsection{Floer homology}
\label{sec:floer}
Let $\varphi=\varphi_H$ be a non-degenerate Hamiltonian
diffeomorphism, which we will view as an element of $\tHam(M)$. Fixing
a ground ring $\F$ (e.g., $\Z_2$ or $\Q$) and an almost complex
structure, both of which will be suppressed in the notation, we denote
by $\CF_*(\varphi)$ and $\HF_*(\varphi)$ the Floer complex and
homology of $\varphi$; see, e.g., \cite{HS, MS, Sa}. The exact
definition of the differential on $\CF_*(\varphi)$ is immaterial for
our purposes, but it is essential that $\CF_*(\varphi)$ is generated
by the capped one-periodic orbits $\bx$ of $H$ and graded by the
Conley--Zehnder index. Furthermore, $\CF_*(\varphi)$ and
$\HF_*(\varphi)$ are filtered by the action of $H$. We have the
canonical isomorphism
\begin{equation}
  \label{eq:Fl-qt}
  \HF_*(\varphi)\cong \HQ_*(M)[-n],
\end{equation}
where $\HQ_*(M)$ is the quantum homology of $M$; see, e.g.,
\cite{Sa,MS} and references therein.

The total homology $\HQ_*(M)$ and $\HF_*(\varphi)$ and the complex
$\CF_*(\varphi)$ are modules over a Novikov ring $\Lambda$, and
$\HQ_*(M)\cong \H_*(M)\otimes \Lambda$ (as a module). There are
several choices of $\Lambda$; see, e.g., \cite{MS}. A specific choice
is inessential for our purposes, but we prefer to think of $\Lambda$
as a certain quotient of the group algebra of $\pi_2(M)$, accounting
for the equivalence of cappings; see, e.g., \cite{HS}. Then $\Lambda$
naturally acts on $\CF_*(\varphi)$ by recapping. We denote by
$|\alpha|$ the degree of $\alpha$ in $\HQ_*(M)$ or
$\HF_*(\varphi)$. Thus the fundamental class $[M]$ has degree $2n$ in
$\HQ_*(M)$ and $n$ in $\HF_*(\varphi)$ and the point class $[\pt]$ has
degree zero in $\HQ_*(M)$ and $-n$ in $\HF_*(\varphi)$.

When $\varphi$ is a pseudo-rotation we have natural isomorphisms
$$
\CF_*(\varphi)\cong \HF_*(\varphi)\cong \HQ_*(M)[-n].
$$
Any iterate $\varphi^k$ is then also a pseudo-rotation, and hence
\begin{equation}
  \label{eq:CF=HF=HQ}
  \CF_*\big(\varphi^k\big) \cong
  \HF_*\big(\varphi^k\big) \cong \HQ_*(M)[-n].
\end{equation}

With the notation set, we are in a position to prove Proposition
\ref{prop:PR-existence}.

\begin{proof}[Proof of Proposition \ref{prop:PR-existence}]
  Let $\varphi$ be a pseudo-rotation of $M^{2n}$. By passing to
  an iterate of $\varphi$ if necessary, we can ensure that for every
  fixed point $x$ of $\varphi$ all elliptic eigenvalues of
  $D\varphi_{x}\colon T_{x}M\to T_{x}M$ are close to $1$ and that for
  any capping of $x$ the loop part of $D\varphi_{x}$, viewed as an
  element of $\tSp(T_{x}M)$, is divisible by $2N$; see Section
  \ref{sec:loop} for the definition.

  By \eqref{eq:CF=HF=HQ}, $\HF_*\big(\varphi^k\big)$ and hence
  $\CF_*\big(\varphi^k\big)$ are supported within $[-n,\,n]+2N\Z$,
  i.e., for any fixed point $\bx^k$ of $\varphi^k$ with any capping
  its Conley--Zehnder index is within this union of the intervals
  $[-n,\,n]+2Nj$. Without loss of generality we may assume that
  $N\geq n+1$; for $n+1\leq 2n$. As a consequence,
  these intervals are disjoint.

  Let $\bx$ be a fixed point of $\varphi$ capped so that
  $\mu(\bx)=n$. Such a point necessarily exists because
  $\HQ_{2n}(M)\neq 0$. Since $x$ is non-degenerate, $\hmu(x)>0$ and
  thus $\mu\big(\bx^{k}\big)\to\infty$; cf.\ \cite{SZ}. Furthermore,
  $\loopp(D\varphi_{\bx})=0$ as is not hard to see from the condition that
  $N\geq n+1$. (Otherwise we would have $\mu(\bx)\geq 2N-n>n$.) Then
  $$
  0\leq \mu\big(\bx^{k+1}\big)-\mu\big(\bx^{k}\big) \leq 2n.
  $$
  (These inequalities follow from Proposition \ref{prop:CZ-qm} and
  also are easy to prove directly.) Therefore, we have
  $$
  \mu\big(\bx^{k}\big)\in [n+1,\,3n]
  $$
  for some $k\in\N$. For this value to be in the support of the Floer
  homology, we must have $2N-n\leq 3n$, i.e., $N\leq 2n$.
 \end{proof}

 \subsubsection{Pair-of-pants product}
\label{sec:pair-of-pants}
Recall that for a pair of Hamiltonian diffeomorphisms $\varphi$ and
$\psi$ we have the \emph{pair-of-pants product}
 $$
 \HF_*(\varphi)\otimes \HF_*(\psi)\to\HF_*(\varphi\psi).
 $$
 This product, which we denote by $*$, has degree $-n$, i.e.,
 $|\alpha * \beta|=|\alpha|+|\beta|-n$. Under the identification
 \eqref{eq:Fl-qt}, the pair-of-pants product turns into the quantum
 product on $\HQ_*(M)$ also denoted by $*$, which is a certain
 deformation of the intersection product on $\H_*(M)$ and has degree
 $-2n$. As a consequence, setting $\varphi=\varphi_1\ldots\varphi_r$,
 we also have the pair-of-pants product
 $$
 \HF_*(\varphi_1)\otimes\ldots \otimes
 \HF_*(\varphi_r)\to\HF_*(\varphi),
 $$
 which agrees with the quantum product, and, in particular,
 $$
  \HF_*\big(\varphi^{k_1}\big)\otimes\ldots\otimes
  \HF_*\big(\varphi^{k_r}\big)
  \to\HF_*\big(\varphi^{k}\big)
$$
where $k_1+\ldots+k_r=k$ and
$$
|\alpha_1|+\ldots +|\alpha_r|-|\alpha_1*\ldots *\alpha_r|
= (r-1)n.
$$

Referring the reader to, e.g., \cite{AS,MS,PSS} for a detailed
treatment of the pair-of-pants product (see also \cite{Se:biased} for
a different and more modern approach) and skipping over some nuances,
we only mention here few relevant points.  There are several ways to
describe the pair-of-pants product on the level of complexes and any
of them is suitable for our purposes as long as it respects the action
filtration. (Thus, for instance, the construction from \cite{PSS} does
not meet this requirement, but the one in \cite{AS} does.)

The product 
$$
  \CF_*(\varphi_1)\otimes\ldots \otimes\CF_*(\varphi_r)\to\CF_*(\varphi)
$$
``counts'' the number of solutions $u\colon \Sigma\to M$ of a suitably
defined Floer equation, where the domain $\Sigma$ is the
$(r+1)$-punctured sphere; see, e.g., \cite{AS, MS}. In other words,
consider capped one-periodic orbits $\bx_i$ of $\varphi_i$ and a
capped one-periodic orbit $\by$ of $H$. Let $\CM$ be the moduli space
of such solutions $u$ ``connecting'' $\bx_1,\ldots ,\bx_r$ to
$\by$. The virtual dimension of $\CM$ is
\begin{equation}
  \label{eq:dim}
\dim\CM=\mu(\bx_1)+\ldots+\mu(x_r)-\mu(\by)
-(r-1)n.
\end{equation}
Assume that this dimension is zero. Then $\by$ enters the product
$\bx_1 *\ldots * \bx_r$ with the coefficient equal to the number of
points (counted with signs if $\F\neq \Z_2$) in the moduli space of
such $u$ ``connecting'' $\bx_1,\ldots ,\bx_r$ to $\by$, provided that
a certain regularity condition is met. This condition, which we will
touch upon in the next section, is satisfied for generic maps
$\varphi_i$.

With or without regularity, we necessarily have
$$
\CA_{H_1}(\bx_1)+\ldots+\CA_{H_r}(\bx_r)-\CA_{H}(\by)=E(u)\geq 0,
$$
where $E(u)$ is the energy of $u$, the Hamiltonian $H_i$ generates
$\varphi_i$ and $H$ generates $\varphi$; see \cite[Eq.\
(3-18)]{AS}. (The choice of $H$ depends on the Hamiltonians $H_i$.) In
particular, $E(u)=0$ if and only if
\begin{equation}
  \label{eq:action}
\CA_{H_1}(\bx_1)+\ldots+\CA_{H_r}(\bx_r)=\CA_{H}(\by).
\end{equation}
In turn, this is the case if and only if $x_1(0)=\ldots=x_r(0)$, the
loop $y$ is the concatenation of the loops $x_i$ and $u$ maps $\Sigma$
onto $y$. Without loss of generality we may assume that the orbits
$x_i$ are constant; see, e.g., \cite[Sect.\ 2.3]{Gi:CC}. Then
\eqref{eq:action} holds if and only $E(u)=0$ and if and only if $u$ is
a constant map.

If the regularity condition is not satisfied, as is often the case for
$\varphi_i=\varphi^{k_i}$, one replaces the maps $\varphi_i$ by their
small perturbations $\varphi'_i$. Since $\varphi_i$ is non-degenerate
there is a one-to-one correspondence between one-periodic orbits of
$\varphi_i$ and $\varphi'_i$ and also a canonical isomorphism
$\CF_*(\varphi_i)\cong\CF_*(\varphi'_i)$. However, this isomorphism
effects the action filtration.

\subsection{Regularity for zero-energy solutions}
\label{sec:trans}
Our goal in this section is to show that zero index, zero energy
pair-of-pants solutions of the Floer equation are automatically
regular. Thus let $x$ be a strongly non-degenerate one-periodic orbit
of $H$ and let $u\colon \Sigma\to M$ be the zero energy solution
asymptotic to $\bx^{k_1}\ldots \bx^{k_r}$ and $\bx^k$ where
$k_1+\ldots k_r=k$. As has been mentioned above, we may assume that
$x$ is a constant one-periodic orbit, and hence $u$ is a constant
solution of the Floer equation mapping $\Sigma$ to $x$.  Denote by
$D\colon \CE^1\to\CE^0$ the linearized Floer operator along $u$. Here
$\CE^1$ is the space of, say, $W^{1,p}$-sections of $u^*TM$ with $p>1$
and $\CE^0$ is the space of $L^p$-sections. The operator $D$ has the
form $\bar{\p}+S$, where $S$ is an automorphism of $u^*TM$, and is
Fredholm due to the non-degeneracy assumption.

\begin{Proposition}
  \label{prop:ker}
  We have $\ker D=0$.
\end{Proposition}

Proposition \ref{prop:ker} is quite standard and has several
predecessors. A variant of the proposition for Floer cylinders is
established in \cite[Sect.\ 2.3]{Sa} and for closed holomorphic curves
in \cite[Lemma 6.7.6]{MS}. Perhaps the easiest way to prove the
proposition is by adapting the argument from \cite[p.\
971]{Se:prod}. Namely, let us pass to Lagrangian Floer theory by using
the graph construction. Then $D$ turns into the Cauchy--Riemann
operator (with the complex structure in the target space parametrized
by the domain) and a solution $\xi$ of the equation $D\xi=0$ becomes a
zero-energy holomorphic map into $T_xM\oplus T_xM$. Such a curve is
necessarily constant and then $\xi=0$ since $\xi$ is globally in
$W^{1,p}$.

Recall that $u$ is regular when $D$ is onto, i.e., $\coker D=0$, and
that the Fredholm index of $D$ is given by \eqref{eq:dim}:
$$
\dim\ker D-\dim\coker D=\mu\big(\bx^{k_1}\big) +
\ldots+\mu\big(x^{k_r}\big)-\mu\big(\bx^k\big) -(r-1)n.
$$
Thus $\coker D=0$ whenever the index of $D$ is zero and we have proved

\begin{Corollary}
  \label{cor:trans}
  Assume that
  $$
  \mu\big(\bx^{k_1}\big)+\ldots+\mu\big(x^{k_r}\big) -
  \mu\big(\bx^k\big) -(r-1)n=0,
  $$
  i.e., $k_1+\ldots+k_r =k$ is an extremal partition (see Definition
  \ref{def:adm_part}). Then the zero energy solution is automatically
  regular.
\end{Corollary}

\section{From extremal partitions to the quantum product}
\label{sec:adm_part-qp}

\subsection{Extremal partitions: the first look}
\label{sec:adm_part}
The notion central to the combinatorial part of the proof of the main
results is that of an extremal partion.

\subsubsection{Definitions and basic facts}
Fix a path $\Phi\in\tSp(2n)$. For the sake of simplicity, we will
assume that $\Phi$ is elliptic and strongly non-degenerate, i.e., the
iterate end-point $\Phi^k(1)$ is non-degenerate for all $k\in\N$.

\begin{Definition}[Extremal Partitions]
  \label{def:adm_part}
  A partition $k_1+\ldots+k_r=k$, $k_i\in\N$, of length $r$ is said to
  be \emph{extremal} (with respect to $\Phi$) if
\begin{equation}
  \label{eq:adm_part}
  \mu\big(\Phi^{k_1}\big) + \ldots +\mu\big(\Phi^{k_r}\big) -
  \mu\big(\Phi^k\big)=(r-1)n.
\end{equation}
\end{Definition}

We will show that the existence of an extremal partition is equivalent
to Condition \ref{cond:A}; see Proposition
\ref{prop:cond-A}. Deferring a detailed discussion of extremal
partitions to Section \ref{sec:adm_part-prop}, we only mention here
two simple facts. Consider the defect
\begin{equation}
  \label{eq:defect}
  D=D(\Phi_1,\ldots,\Phi_r):=\sum\mu\big(\Phi_i\big) -
  \mu\big(\Phi_1\cdot\ldots\cdot\Phi_r\big)
\end{equation}
of the ``Conley--Zehnder quasimorphism'', where we have assumed that
all $\Phi_i$ and all partial products
$\Phi_1\cdot\ldots\cdot\Phi_\ell$, $\ell\leq r$, are
non-degenerate. Then, as is shown in \cite{DG2P},
$$
|D|\leq (r-1)n.
$$
(The non-degeneracy requirement is essential.) We will further discuss
this fact and give a short proof in Section \ref{sec:adm_part-extra};
see Proposition \ref{prop:CZ-qm}. In particular, for any partition
$k_1+\ldots+k_r=k$, $k_i\in\N$, we have
\begin{equation}
\label{eq:defect2}
\mu\big(\Phi^{k_1}\big)+\ldots+\mu\big(\Phi^{k_r}\big)
- \mu\big(\Phi^k\big)\leq
(r-1)n
\end{equation}
as long as the products are non-degenerate; cf.\ \cite[Lemma
5.10]{Se:prod}. Thus extremal partitions maximize the defect; hence,
the term.

Furthermore, $D$ depends only on the end-points
$\Phi_1(1),\ldots,\Phi_r(1)$. Indeed, composing one of the maps
$\Phi_i$ with a loop changes both terms in \eqref{eq:defect} by the
mean index of the loop. In particular, the left-hand side of
\eqref{def:adm_part} is completely determined by $\Phi(1)$ and, of
course, the partition. In other words, whether or not
$k_1+\ldots+k_r=k$ is an extremal partition is a feature of $\Phi(1)$,
but the indices $\mu\big(\Phi^{k_i}\big)$ depend on the path $\Phi$.
For the sake of brevity we set
$\Gamma(\Phi):=\Gamma\big(\Phi(1)\big)$.

\begin{Example}
  \label{ex:ccwise}
  Assume that $\Phi$ is the direct sum of $n$ counterclockwise
  rotations $\exp\big(2\pi\sqrt{-1}\lambda_i t\big)$, where
  $\lambda_i>0$ are small and $t\in [0,\,1]$. Then
  $\mu\big(\Phi^r\big)=n$ as long as $r\max\lambda_i<1$, and
  $1+\ldots+1=r$ is an extremal partition with \eqref{eq:adm_part}
  taking form $rn-(r-1)n=n$.
\end{Example}

\begin{Example}
  Assume that $\Phi$ is toric, i.e., $\dim\Gamma(\Phi)=n$. Then $\Phi$
  admits extremal partitions of arbitrarily large length. We will
  prove this fact in Section \ref{sec:adm_part-prop}, but it is also
  not hard to see this directly as a consequence of Example
  \ref{ex:ccwise}.
\end{Example}

\begin{Example}
  \label{ex:sum}
  Assume that $\Phi$ is the sum of a clockwise rotation
  $\exp\big(-2\pi\sqrt{-1}\lambda t\big)$, $t\in [0,\,1]$, where
  $\lambda>0$, and the counterclockwise rotation in the same
  angle. Then $\mu\big(\Phi^k\big)=0$ for all $k\in\N$ and $\Phi$ does
  not admit extremal partitions.
\end{Example}

\begin{Example}
  \label{ex:clwise}
  Assume that $\Phi$ is the clockwise rotation
  $\exp\big(-2\pi\sqrt{-1}\lambda t \big)$, $t\in [0,\,1]$, where
  $\lambda>0$, or the direct sum of such rotations by the same
  angle. Then $\Phi$ also admits extremal partitions for any
  $\lambda$. This follows, for instance, from Example \ref{ex:ccwise}
  and Proposition \ref{prop:extr-part-prop} or can be verified
  directly.
\end{Example}

\begin{Remark}
  \label{rmk:elliptic}
  The condition that $\Phi$ is elliptic imposed above is in some sense
  redundant: non-elliptic symplectic maps simply do not admit extremal
  partitions. This fact readily follows from Proposition
  \ref{prop:extr-part-prop} and is also easy to prove directly.
\end{Remark}

\subsubsection{Combinatorial results: the existence of extremal partitions}
As one can guess already from \eqref{eq:dim} giving the dimension of
the relevant moduli spaces, extremal partitions are intimately related
to certain products in quantum homology; see Theorem
\ref{thm:general}. However, to conclude from this that the quantum
product is deformed one needs to have additional information about the
classes involved in the product, which in our context is a
combinatorial problem. For instance, the proof of Theorem
\ref{thm:main} hinges on the following result.

\begin{Theorem}[Extremal Partition Theorem]
  \label{thm:adm_part}
  Let $\Phi\in\tSp(2n)$ be elliptic and strongly
  non-degenerate. Assume that for some $r\in \N$ the linear symplectic
  map $\Phi(1)$ satisfies Condition \ref{cond:A} and also, for some
  $N\in\N$, Condition \ref{cond:B1} or Condition \ref{cond:B2}. Then
  there exists an extremal partition $k_1+\ldots+k_r=k$ with respect
  to $\Phi$ such that
\begin{equation}
  \label{eq:adm_part-index}
  \mu\big(\Phi^{k_i}\big)\not\equiv n\mod 2N
  \textrm{ for all } i=1,\ldots, r.
\end{equation}
\end{Theorem}

Note that here, as in Theorem \ref{thm:main}, we could have required
that $N\geq 2$, since Conditions \ref{cond:B1} and \ref{cond:B2} are
never satisfied when $N=1$. It is also worth pointing out again that
in this theorem Conditions \ref{cond:A} and \ref{cond:B1} or
\ref{cond:B2} play very different roles. Condition \ref{cond:A} is
necessary and sufficient to guarantee the existence of an extremal
partition (cf.\ Proposition \ref{prop:cond-A}), while Condition
\ref{cond:B1} or \ref{cond:B2} is used to establish
\eqref{eq:adm_part-index}.

In a similar vein, Theorem \ref{thm:toric} relies on the combinatorics
of extremal partitions in the toric case.

\begin{Theorem}
  \label{thm:adm_part-dense}
  Assume that $\Phi$ is toric, i.e., $\Gamma(\Phi)=\T^n$. Then, for
  every $r\geq 1$, there exists an extremal partition $m+\ldots+m =k$
  of length $r$ (i.e., $r\cdot m=k$) such that
  \begin{equation}
    \label{eq:deg-dense}
    \mu\big(\Phi^{m}\big) \equiv n-2 \mod 2N.
  \end{equation}
\end{Theorem}

Note that $\Phi$ is then strongly non-degenerate and all eigenvalues
of $\Phi(1)$ are necessarily distinct; cf.\ Example
\ref{ex:base_grp2}. In particular, $\Phi(1)$ is automatically
semi-simple if $\Gamma=\T^n$.

\begin{Remark}
\label{rmk:connected}
An immediate consequence of the proof of Theorem
\ref{thm:adm_part-dense} is that the assertion of the theorem also
holds whenever Conditions \ref{cond:A} and \ref{cond:B2} are satisfied
for the connected component of the identity $\Gamma_0(\Phi)$. (The
same is true for Theorem \ref{thm:main}.) We will use this fact in
Section \ref{sec:adm_part-pf} the proof of Theorem \ref{thm:pr-dim4}.
\end{Remark}

Finally, Theorem \ref{thm:pr-dim4} is also a consequence of the
following combinatorial result.

\begin{Theorem}
  \label{thm:adm_part-dim4}
  Let $\Phi\in\tSp(4)$ be elliptic and strongly non-degenerate, and
  such that $\dim\Gamma(\Phi)=1$.  Assume furthermore that the slope
  $s\neq\pm 1,\, 3,\, 1/3,\, -2,\, -1/2$ when $N=2$ or
  $s\neq -1, \, \pm 1/2, \, \pm 2$ when $N= 3$. Then there exists
  an extremal partition of length $3$ such that
  \begin{equation}
    \label{eq:adm_part-dim4}
  \mu\big(\Phi^{k_i}\big) \not\equiv 2 \mod 2N \textrm{ for
    $i=1,\,2,\,3$.}
  \end{equation}
\end{Theorem}

\begin{Remark}
  This theorem is more precise than Theorem \ref{thm:adm_part} and it
  gives essentially a necessary and sufficient condition in dimension
  four. Namely, assume that $\Gamma$ is connected and its slope is
  ``black-listed'' in Theorem \ref{thm:adm_part-dim4}. Then there
  exists $\Phi\in\tSp(4)$ such that $\Gamma(\Phi)=\Gamma$ and there
  are no extremal partitions satisfying
  \eqref{eq:adm_part-dim4}. However, in the setting of the theorem,
  $\Phi$ still satisfies Condition \ref{cond:A}, and if Condition
  \ref{cond:B2} holds the desired partitions exist.
\end{Remark}

The conditions of these theorems are satisfied for most (but not all)
of strongly non-degenerate, elliptic $\Phi\in \tSp(2n)$.

\subsection{Combinatorics of extremal partitions and the quantum
  product}
\label{sec:mainthm-pf}
In this section we establish the main result of the paper, Theorem
\ref{thm:main}, as an easy consequence of Theorem \ref{thm:adm_part}
and Corollary \ref{cor:trans}, and also Theorems \ref{thm:toric} and
\ref{thm:pr-dim4}. The proofs of all three theorems follow the same
path and can be rephrased as a general argument reducing the problem
to a combinatorial question.

Recall that since $\varphi$ is a pseudo-rotation, for every $k\in\N$
we have canonical identifications \eqref{eq:CF=HF=HQ}:
$$
\CF_*\big(\varphi^k\big)\cong\HF_*\big(\varphi^k\big)\cong\HQ_*(M)[-n],
$$
where we view $\varphi$ as an element of $\tHam(M)$ rather than
$\Ham(M)$.

\begin{Theorem}
  \label{thm:general}
  Let $\bx$ be a capped one-periodic orbit of a pseudo-rotation
  $\varphi$, and let $k_1+\ldots+ k_r=k$ be an extremal partition of
  length $r$ with respect to $\Phi:=D\varphi^t|_{\bx}$. Using
  \eqref{eq:CF=HF=HQ}, set $\alpha_i=[\bx^{k_i}]\in\HQ_*(M)$.  Then
  $|\alpha_i|=n+\mu\big(\Phi^{k_i}\big)$ and \eqref{eq:prod} holds:
$$
\alpha_1*\ldots*\alpha_r\neq 0.
$$
\end{Theorem}

Note that in the setting of this theorem $x$ is automatically
elliptic; see Remark \ref{rmk:elliptic}. Theorems \ref{thm:main} and
\ref{thm:pr-dim4} immediately follow from this general result and
Theorems \ref{thm:adm_part} and \ref{thm:adm_part-dim4}, combined with
the observation that all iterated indices $\mu\big(\Phi^k\big)$ have
the same parity when $\Phi$ is elliptic.

\begin{proof}[Proof of Theorem \ref{thm:general}]
  The argument is based on Example \ref{ex:irr-rot}. Clearly
  $|\alpha_i|=n+\mu\big(\Phi^{k_i}\big)$. Thus we only need to verify
  \eqref{eq:prod}, i.e., that
\begin{equation}
  \label{eq:x-prod}
[\bx^{k_1}]*\ldots *[\bx^{k_r}]\neq 0,
\end{equation}
where we have now identified the quantum product with the
pair-of-pants product
$$
\HF_*\big(\varphi^{k_1}\big)\otimes\ldots\otimes
\HF_*\big(\varphi^{k_r}\big) \to \HF_*\big(\varphi^{k}\big).
$$

Consider small non-degenerate perturbations $\varphi_{k_i}$ of
$\varphi^{k_i}$ such that on the level of Floer complexes the
regularity condition is satisfied for the pair-of-pants product
$$
\CF_*\big(\varphi_{k_1}\big)\otimes\ldots\otimes
\CF_*\big(\varphi_{k_r}\big) \to \CF_*\big(\varphi_{k}\big),
$$
where $\varphi_k:=\varphi_{k_r}\circ\ldots\circ \varphi_{k_1}$. Note
that $\varphi_k$ is also a small perturbation of $\varphi^k$, and we
have have canonical isomorphisms of the Floer complexes
$$
\CF_*\big(\varphi_{k_i}\big)=\CF_*\big(\varphi^{k_i}\big)
\textrm{ and }
\CF_*\big(\varphi_k\big)=\CF_*\big(\varphi^k\big).
$$
Furthermore, by Corollary \ref{cor:trans}, we can make these
perturbations such that $\varphi_{k_i}=\varphi^{k_i}$ near $x$ and, as
a consequence, $\varphi_k=\varphi^k$ on a small neighborhood of
$x$. (Here it is convenient to assume that $x$ is a constant
one-periodic orbit -- this can always be achieved by composing
$\varphi$ with a contractible loop; see, e.g., \cite[Sect.\
2.3]{Gi:CC}.) Thus $\bx^{k_i}$ is still a capped one-periodic orbit of
$\varphi_{k_i}$ and $\bx^k$ is a capped periodic orbit of
$\varphi_k$. Let us redenote these orbits as $\bx_{k_i}$ and $\bx_k$,
respectively.

The only zero-energy pair-or-pants curves are constant; see Section
\ref{sec:pair-of-pants}. Thus the constant curve is the only curve
from $(\bx_{k_1},\ldots,\bx_{k_r})$ to $\bx_k$. Furthermore, consider
the modular space of such curves. This modular space has virtual
dimension zero and the constant curve from
$(\bx_{k_1},\ldots,\bx_{k_r})$ to $\bx_k$ is regular by Corollary
\ref{cor:trans}.  Therefore,
\begin{equation}
  \label{eq:x-prod2}
\bx_{k_1}*\ldots *\bx_{k_r} =\bx_k+\ldots,
\end{equation}
where the dots stand for capped periodic orbits of $\varphi_k$ with
action strictly smaller than the action of $\bx_k$. As a consequence,
$$
[\bx^{k_1}]*\ldots *[\bx^{k_r}]=[\bx^k]+\ldots,
$$
where the dots represent again some cohomology classes generated by
the orbits with action strictly smaller than the action of
$\bx^k$. Hence, the right-hand side is non-zero. This proves
\eqref{eq:x-prod} and concludes the proof of the theorem.
\end{proof}

\begin{Remark}
  Choosing the perturbations $\varphi_{k_i}$ equal to $\varphi^{k_i}$
  near $x$ is convenient but not really necessary. Since the constant
  pair-of-pants curve from $(\bx^{k_1},\ldots,\bx^{k_r})$ to $\bx^k$
  is regular, it will persist under a small perturbation turning into
  one non-constant small energy curve. This is enough to separate the
  action of $\bx_k$ from the actions of other periodic orbits on the
  right-hand side of the product \eqref{eq:x-prod2}; cf.\ \cite[Prop.\
  2.2]{GG:Rev}.
\end{Remark}

\begin{proof}[On the proof of Theorem \ref{thm:toric}]
  The result easily follows from Theorem \ref{thm:general}. We only
  need to make sure that in this case we can take the product of $r$
  equal elements of degree $2n-2$.  Let $x$ be a one-periodic orbit of
  $\varphi$ such that $\dim\Gamma(x)=n$. Let $\Phi=D\varphi^t|_{\bx}$,
  where we have used an arbitrary capping of $x$, and let $m$ be as in
  Theorem \ref{thm:adm_part-dense}. Then $m+\ldots+m=rm$ is an
  extremal partition for $\Phi$. Although in general the degree of
  $[\bx^m]$ need not be equal to $2n-2$, we have $\mu(x^m)= n-2$ in
  $\Z_{2N}$ by \eqref{eq:deg-dense}. (Recall that the Conley--Zehnder
  index of an un-capped orbit is well-defined as an element of
  $\Z_{2N}$.)  Denote by $\by$ the orbit $x^m$ capped so that
  $\mu(\by)=n-2$. Then $[\by]^r=f\cdot [\bx^m]^r$ for some $f\neq 0$
  in the Novikov ring, and $[\bx^m]^r\neq 0$ by Theorem
  \ref{thm:general}.  It follows that $\alpha^r\neq 0$ and
  $|\alpha|=2n-2$, where $\alpha=[\by]$.
\end{proof}

\section{Study of extremal partitions}
\label{sec:adm_part-study}

\subsection{General properties}
\label{sec:adm_part-prop}
The notion of an extremal partition is certainly interesting by
itself. In this section we establish some of their general properties,
recalling some of the facts already mentioned in Section
\ref{sec:adm_part} and going slightly father than is strictly speaking
necessary for applications to our main results. We start with the
following general result concerning the defect of ``the
Conley--Zehnder quasimorphism'':

\begin{Proposition}[Cor.\ 3.5 in \cite{DG2P}]
\label{prop:CZ-qm}
For any two non-degenerate elements $\Phi$ and $\Psi$ of $\tSp(2n)$,
we have
$$
\big|\mu(\Psi\Phi)-\mu(\Psi)-\mu(\Phi)\big|\leq n.
$$
\end{Proposition}
This upper bound is sharp and the non-degeneracy requirement is
essential. The proposition in particular implies a sharp upper bound
for the defect of several Conley--Zehnder (or Maslov-) type
quasimorphisms on $\tSp(2n)$; see Remark \ref{rmk:Maslov-qm}. For the
sake of completeness, we give a short and elementary proof of the
proposition in Section \ref{sec:adm_part-extra}. (Note also that the
regularity arguments from Section \ref{sec:trans} can be turned into
an analytical proof of Proposition \ref{prop:CZ-qm}.)

As a consequence, for any partition $k_1+\ldots+k_r=k$, $k_i\in\N$, we
have \eqref{eq:defect2}, i.e.,
$$
\mu\big(\Phi^{k_1}\big)+\ldots+\mu\big(\Phi^{k_r}\big) -
\mu\big(\Phi^k\big)\leq (r-1)n,
$$
as long as the products are non-degenerate. (Another way to state this
fact is that the function $k\mapsto \mu\big(\Phi^k\big)-n$ is
sub-additive.) Thus extremal partitions maximize the left-hand side of
this inequality. This upper bound is again sharp and non-degeneracy is
essential.

\begin{Proposition}[Properties of Extremal Partitions]
  \label{prop:extr-part-prop}
  Let $\Phi\in\tSp(2n)$ and $\Psi\in\tSp(2n')$.
\begin{itemize}
\item[\rm{(i)}] A partition $k_1+\ldots+k_r=k$ is extremal for $\Phi$
  if and only if it is extremal for $\phi\Phi$ for any loop $\phi$ in
  $\Sp(2n)$. Thus the property to be extremal depends only on
  $\Phi(1)\in\Sp(2n)$.
\item[\rm{(ii)}] A partition $k_1+\ldots+k_r=k$ is extremal for
  $\Phi^m$ if and only if $mk_1+\ldots+mk_r=mk$ is extremal for
  $\Phi$.
\item[\rm{(iii)}] A partition $k_1+\ldots+k_r=k$ is extremal for
  $\Phi\oplus\Psi$ if and only if it is simultaneously extremal for
  $\Phi$ and $\Psi$.
\item[\rm{(iv)}] Assume that $k_1+\ldots+k_r=k$ and
  $\ell_1+\ldots+\ell_s=k_1$ are extremal partions for $\Phi$. Then
  $\ell_1+\ldots+\ell_s+k_2+\ldots+k_r=k$ is also an extremal
  partition for $\Phi$. Conversely, assume that $k_1+\ldots+k_r=k$ is
  an extremal partition. Then for any $1\leq s\leq r$, the sum
  $k_1+\ldots+k_s=:m$ of the first $s$ terms and the sum
  $m+k_{s+1}+\ldots+k_r=k$ are also extremal partitions.
 \end{itemize}
\end{Proposition}

\begin{proof} 
  Recall that $k_1+\ldots+k_r=k$ is an extremal partition for $\Phi$
  if \eqref{eq:adm_part} holds:
$$
\mu\big(\Phi^{k_1}\big)+\ldots+\mu\big(\Phi^{k_r}\big) -
\mu\big(\Phi^k\big)=(r-1)n.
$$
Replacing $\Phi$ by $\phi\Phi$ adds $k_i\hmu(\phi)$ and $k\hmu(\phi)$
to the terms on the left and hence does not effect the sum. This
proves (i). Assertion (ii) is obvious from the definition.

By additivity of the Conley--Zehnder index, an extremal partition of
$\Phi$ and $\Psi$ is also an extremal partition for $\Phi\oplus
\Psi$. Conversely, if a partion is not extremal for $\Phi$ or/and
$\Psi$, \eqref{eq:adm_part} becomes a strict inequality by Proposition
\ref{prop:CZ-qm}. Adding up these inequalities for $\Phi$ and $\Psi$
we obtain a strict inequality for $\Phi\oplus \Psi$. This concludes
the proof of (iii).

In one direction, Assertion (iv) is also clear from the definition. To
prove the converse, consider an extremal partition
$k_1+\ldots+k_r=k$. For $s<r$, set $m=k_1+\ldots+k_s$. We need to show
that the partitians $k_1+\ldots+k_s=m$ and $m+k_{s+1}+\ldots+k_r=k$
are also extremal. Assume not. Then, by Proposition \ref{prop:CZ-qm},
we have
$$
\mu\big(\Phi^{k_1}\big)+\ldots+\mu\big(\Phi^{k_s}\big)-
\mu\big(\Phi^m\big)\leq(s-1)n
$$
and
$$
\mu\big(\Phi^{m}\big)+\mu\big(\Phi^{k_{s+1}}\big)+
\ldots+\mu\big(\Phi^{k_r}\big)-\mu\big(\Phi^k\big)\leq(r-s)n,
$$
where at least one of the inequalities is strict by the
assumption. Combining these inequalities, we conclude that
\eqref{eq:adm_part} is also strict and thus the original partition is
not extremal.
\end{proof}

The role of Condition \ref{cond:A} in our method is clarified by the
next result.

\begin{Proposition}
  \label{prop:cond-A}
  An elliptic element $\Phi\in\tSp(2n)$ admits an extremal partition
  of length $r$ if and only if $\Gamma(\Phi)$ satisfies Condition
  \ref{cond:A}.
\end{Proposition}

We emphasize that both Condition \ref{cond:A} and the existence of
extremal partitions are in fact properties of $\Phi(1)\in\Sp(2n)$; see
Proposition \ref{prop:extr-part-prop} (i).

\begin{Example}[Toric $\Phi$ revisited]
  \label{ex:toric}
  Assume that $\dim\Gamma=n$, i.e., $\Gamma=\T^n$. Then $\Phi$ admits
  extremal partions of arbitrarily large length; cf.\ Example
  \ref{ex:base_grp2}.
\end{Example}

\begin{proof}[Proof of Proposition \ref{prop:cond-A}]
  The sequence of iterated indices $\mu(\Phi^k)$, $k\in\N$, does not
  change under an isospectral deformation of $\Phi$. Hence we can
  require $\Phi(1)$ to be semi-simple and view it as a topological
  generator of $\Gamma$.

  Assume first that Condition \ref{cond:A} is satisfied: there exist
  $r$ points $\vtheta_1,\ldots,\vtheta_r$ in $\Gamma$ such that
  \eqref{eq:cond_A} holds:
 $$
  \sum_{i=1}^r \lambda_{ij}<1\textrm{ for all $j=1,\ldots, n$},
 $$
 where
 $ \vtheta_i=\big(e^{2\pi\sqrt{-1}\lambda_{i1}},\ldots,
 e^{2\pi\sqrt{-1}\lambda_{in}}\big)$ with $0<\lambda_{ij}<1$. For
 $i+1,\ldots, r$, set
 \begin{equation}
  \label{eq:short_paths}
  \Psi_i(t)=\big(e^{2\pi\sqrt{-1}\lambda_{i1}t},\ldots,
  e^{2\pi\sqrt{-1}\lambda_{in}t}\big), \quad t\in [0,\,1].
 \end{equation}
 Then $\mu(\Psi_i)=n$ and also $\mu(\Psi_1\ldots\Psi_r)=n$ by
 Condition \ref{cond:A}. Therefore,
 $$
 \sum\mu(\Psi_i)-\mu(\Psi_1\ldots\Psi_r)=(r-1)n.
 $$
 The end-points $\vtheta_i=\Psi_1(1)$ can be approximated arbitrarily
 well by the iterates $\Phi^{k_i}(1)$ for some $k_i\in\N$. In other
 words, for a loop $\phi_i$, the element $\phi_i\Phi^{k_i}$ can made
 arbitrarily close to $\Psi_i$ and the product of $\phi_i\Phi^{k_i}$
 can be made arbitrarily close to $\Psi_1\ldots\Psi_r$. Note that that
 this product has the form $\phi\Phi^{k}$, where $k=k_1+\ldots+k_r$
 and $\phi$ is the product of the loops $\phi_i$. (The group
 $\pi_1\big(\Sp(2n)\big)$ is in the center of $\tSp(2n)$.) Recalling
 that the defect depends only on the end-points, we see that
 $$
 \sum\mu\big(\Phi^{k_i}\big)-\mu\big(\Phi^k\big) =
 \sum\mu(\Psi_i)-\mu(\Psi_1\ldots\Psi_r)=(r-1)n.
 $$

 Conversely, assume that $k_1+\ldots+k_r=k$ is an extremal partition
 for $\Phi$. Set $\vtheta_i=\Phi^{k_i}(1)$. It suffices to show that
 \eqref{eq:cond_A} holds, where $\lambda_{ij}$ are as above. Define
 the paths $\Psi_i$ by \eqref{eq:short_paths} and set
 $\Psi=\Psi_1\ldots\Psi_r$. (By slightly perturbing $\Phi$ we can
 ensure that $\Psi$ is non-degenerate.) Then $\mu(\Psi_i)=n$ and
 $$
 \sum\mu(\Psi_i)-\mu(\Psi)=
 \sum\mu\big(\Phi^{k_i}\big)-\mu\big(\Phi^k\big) =(r-1)n
 $$
 by \eqref{eq:adm_part}. Thus $\mu(\Psi)=n$.

 Clearly,
 $$
 \Psi(t)=\big(e^{2\pi\sqrt{-1}\lambda_{1}t},\ldots,
 e^{2\pi\sqrt{-1}\lambda_{n}t}\big), \quad t\in [0,\,1],
 $$
 where
 $$
\lambda_j= \sum_{i=1}^r \lambda_{ij}>0.
 $$
 All intersection points of this path with the Maslov cycle are
 positive. The condition that $\lambda_j<1$ for all $j$ is equivalent
 to that the only intersection of $\Psi$ with the Maslov cycle is at
 $t=0$ which, in turn, is equivalent to that $\mu(\Psi)=n$.
\end{proof}

Condition \ref{cond:A} is hard to visualize and verify directly and
this is where the following criterion comes handy. Denote by $\Pi_r$
the open cube $(0,\, 1/r)^n$ in the torus $\T^n$ identified with the
quotient of the cube $[0,\,1]^n$.

\begin{Proposition}
  \label{prop:cond-A2}
  Assume that $\codim \Gamma\leq 1$. Then Condition \ref{cond:A} is
  satisfied for $\Gamma$ if and only if
  $\Gamma\cap \Pi_r\neq \emptyset$.
\end{Proposition}

In dimension four, this gives a general necessary and sufficient
condition for Condition \ref{cond:A} to be satisfied:

\begin{Corollary}
  \label{cor:cond-A2}
  Assume that $n=2$ and $\dim \Gamma\geq 1$. Then Condition
  \ref{cond:A} is satisfied if and only if
  $\Gamma\cap \Pi_r\neq \emptyset$.
\end{Corollary}

Here of course the case of $\dim\Gamma =1$ is most interesting: when
$\dim\Gamma =2$, Condition \ref{cond:A} obviously holds; see Example
\ref{ex:toric}.

\begin{Remark}
  \label{rmk:Cond-A-2D}
  If $n=2$, for every $r$ we have $\Gamma\cap \Pi_r\neq \emptyset$ for
  all but a finite number of subgroups $\Gamma\subset \T^2$ of
  positive dimension. For instance, assume that $\Gamma$ is connected
  and $r=3$ -- this is the minimal value of $r$ needed in dimension
  four to detect the quantum product. Then
  $\Gamma\cap \Pi_3=\emptyset$ if and only if the slope
  $s=-1, \, -2, \, -1/2$.  This is no longer true when $n\geq 3$, but
  even then this requirement is met for a majority of subgroups.
\end{Remark}

\begin{proof}[Proof of Proposition \ref{prop:cond-A2}]
  In one direction the assertion is obvious and requires no additional
  conditions on $\Gamma$. Thus we need to show that
  $\Gamma\cap \Pi_r\neq \emptyset$ whenever Condition \ref{cond:A}
  holds.

  Throughout the proof we will view the cube $C=[0,\,1]^n\subset \R^n$
  as the fundamental domain for $\T^n$. Let $H$ be the inverse image
  of the connected component of the identity in $\R^n$ and $L_q$ stand
  for connected components of the inverse image $L$ of $\Gamma$ in $C$
  under the natural maps $C\to\T^n$. Clearly, each $L_q$ is the
  intersection of a hyperplane parallel to $H$ with $C$. Among these
  denote by $L_0$ be the component closest to $0$.

  If $0\in L_0$, we obviously have $L_0\cap \Pi_r\neq \emptyset$ for
  all $r$ and the proof is finished. Thus we can assume that
  $0\not\in L_0$, i.e., $L_0$ is a positive distance from $0$.
  
  Denote by $\vl_i\in C$ the inverse image of the point $\vtheta_i$
  from Condition \ref{cond:A}. Thus, in the notation for that
  condition,
    $$
    \vl_i=(\lambda_{i1},\ldots,\lambda_{in})
    $$
  and \eqref{eq:cond_A} holds for all $j=1,\ldots, n$, i.e.,
    $$
    \sum_{i=1}^r \lambda_{ij}<1.
    $$
  Note that the points $\vl_i$ may lie on different components $L_q$.

  Consider the segment $Y_i=\{t\vl_i\mid t\in [0,\,1]\}$ connecting
  $0$ and $\vl_i$. The intersection $\vl'_i\in Y_i\cap L$ that is
  closest to zero lies on $L_0$. Denote the components of $\vl'_i$ by
  $\l'_{ij}$. Then $\l'_{ij}=t\l_{ij}$ with $t\in (0,\,1]$, and hence
    $$
    \sum_{i=1}^r \lambda'_{ij}\leq \sum_{i=1}^r \lambda_{ij}<1.
    $$
  Let
    $$
    \vl'=\frac{1}{r}\sum_{i=1}^r\vl'_i.
    $$
    to be the mean or `the ``center of mass'' of the points $\vl'_i$.
    Then $\vl'\in L_0$, since all $\vl'_i\in L_0$ and $L_0$ is
    convex. As a consequence, the projection $\vtheta$ of $\vl'$ to
    $\T^n$ is in $\Gamma$. Furthermore, for every component $\l'_j$ of
    $\vl'$ we have
  $$
  \l'_j=\frac{1}{r}\sum_{i=1}^r\l'_{ij}\leq
  \frac{1}{r}\sum_{i=1}^r\l_{ij}<\frac{1}{r}.
  $$
  Therefore, $\vtheta'\in \Gamma\cap\Pi_r$.
\end{proof}

\begin{Remark}
  It is very unlikely that Proposition \ref{prop:cond-A2} would
  hold in other settings without significant constraints on $\Gamma$
  and $r$. However, some partial results are certainly feasible.  For
  instance, assuming that $\Gamma$ is connected, one could expect the
  proposition to hold, perhaps under some additional (un-)divisibility
  conditions on $\mu_\Gamma$ and~$r$.
 \end{Remark}

\subsection{Proofs of Theorems \ref{thm:adm_part},
  \ref{thm:adm_part-dense} and \ref{thm:adm_part-dim4}}
\label{sec:adm_part-pf}
In this section we establish the combinatorial results underlying the
main theorems of the paper.

\begin{proof}[Proof of Theorem \ref{thm:adm_part}] As in the proof of
  Proposition \ref{prop:cond-A}, we can require $\Phi$ to be
  semi-simple; thus $\Phi(1)\in\Gamma$. Throughout the proof we will
  assume that Condition \ref{cond:A} is satisfied. Thus, by
  Proposition \ref{prop:cond-A}, there exists an extremal partition
  $\ell_1+\ldots+\ell_r=\ell$ of length $r$. Our goal is to modify it
  if necessary, creating a new extremal partition $k_1+\ldots+k_r=k$
  such that \eqref{eq:adm_part-index} holds:
  $$
  \mu\big(\Phi^{k_i}\big)\not\equiv n\mod 2N \textrm{ for all
    $i=1,\ldots, r$.}
  $$

  Assume first that Condition \ref{cond:B1} is satisfied, i.e., there
  exists a loop $\gamma$ in $\Gamma$ such that
  $N\not|\, \mu_\Gamma(\gamma)=\hmu(\gamma)/2$. Here we view $\gamma$
  as a loop in $\Sp(2n)$ and, in particular, an element of $\tSp(2n)$.

  We claim that $\gamma$ can be approximated arbitrarily well by the
  elements of the form $\phi\Phi^m$, where $\phi$ is a loop in
  $\Sp(2n)$ with $2N\mid\hmu(\phi)$.

  To prove this, observe that we can take a one-dimensional subgroup
  in $\Gamma$ as $\gamma$. Then there exists an element $\tpsi$ in the
  inverse image of $\Gamma$ in $\tSp(2n)$ such that
  $\tpsi^{2N}=\gamma$. Let $\psi$ be its image in $\Gamma$. Since
  $\Phi(1)$ generates $\Gamma$, we can approximate $\psi$ by the
  powers $\Phi^s(1)$ arbitrarily well. As a consequence, we can
  approximate $\tpsi$ arbitrarily well in $\tSp(2n)$ by the elements
  of the form $\phi_0\Phi^s$ where $\phi_0$ is a loop. Hence, the
  elements $\phi_0^{2N}\Phi^{2Ns}$ approximate $\gamma$. Then
  $\hmu\big(\phi_0^{2N}\big)=2N\hmu(\phi_0)$ and it remains to set
  $\phi=\phi_0^{2N}$ and $m=2Ns$.

  With the claim established, we are ready to modify the partition
  $\ell_1+\ldots+\ell_r=\ell$. If
  $\mu\big(\Phi^{\ell_i}\big) \not\equiv n\mod 2N$ we simply set
  $k_i=\ell_i$. If $\mu\big(\Phi^{\ell_i}\big) \equiv n\mod 2N$ we
  replace $\ell_i$ by $k_i=\ell_i+m$. Then
  $$
  \mu\big(\Phi^{k_i}\big)=\mu\big(\Phi^{\ell_i}\Phi^m\big).
  $$
  Making the approximation accurate enough, we have
  $$
  \mu\big(\Phi^{\ell_i}\Phi^m\big)=
  \mu\big(\Phi^{\ell_i}\phi^{-1}\gamma\big)=
  \mu\big(\Phi^{\ell_i}\big)+\hmu\big(\phi^{-1}\big)+\hmu(\gamma).
  $$
  Here the first term is congruent to $n$ modulo $2n$, the second term
  is divisible by $2N$ and the last term is not divisible by
  $2N$. Thus $\mu\big(\Phi^{k_i}\big) \not\equiv n\mod 2N$.

  It remains to show that the new partition is still extremal. The
  modification results in replacing $\Phi^{\ell_i}$ by $\Phi^{k_i}$
  which is approximately a product of $\Phi^{\ell_i}$ with a loop and
  likewise $\Phi^k$ is approximately the product of $\Phi^\ell$ with a
  loop. It is clear that if the approximations are accurate enough,
  depending only on $\Phi^{\ell_i}(1)$ and their products, the
  partition will remain extremal.

  Next assume that Condition \ref{cond:B2} is satisfied: $\Gamma$ is
  connected and there exists a convex neighborhood $V$ of $0\in \T^n$
  whose intersection with $\Gamma$ is connected and an iterate
  $\Phi^k(1)\in V$ such that $2N\not|\,\loopp\big(\Phi^k\big)$.
  In addition, we can also require that $N\mid\mu_\Gamma$,
  i.e., Condition \ref{cond:B1} fails.

  Since $\Gamma$ is connected and Condition \ref{cond:A} is a feature
  of $\Gamma$, replacing the original $\Phi$ by $\Phi^k$, we can
  assume that $\ell_1+\ldots+\ell_r=\ell$ is an extremal partition for
  $\Phi$ where $\Phi(1)\in V\cap\Gamma$ and $2N\not|\,\loopp(\Phi)$
  and $2N\mid \ell_i$ for all $i$.
  This is the partition we will change to a new extremal partition
  $k_1+\ldots+k_r=k$ such that \eqref{eq:adm_part-index} holds.

  As in the first part of the proof, we set $k_i=\ell_i$, i.e., no
  modification is needed, if
  $\mu\big(\Phi^{\ell_i}\big) \not\equiv n\mod 2N$. In the rest of the
  argument we describe how to change $\ell_i$ when
  $\mu\big(\Phi^{\ell_i}\big) \equiv n\mod 2N$.

  Since $\Gamma$ is connected, for every $i=1,\ldots, r$, any
  arithmetic progression contains an infinite subsequence
  $k_{ij}\to\infty$ such that $\Phi^{k_{ij}}(1)\to\Phi^{\ell_i}(1)$ as
  $j\to\infty$. Thus, setting $k_i=k_{ij}$, we will assume in what
  follows that $\Phi^{k_i}(1)$ is sufficiently close to
  $\Phi^{\ell_i}(1)$.

  We claim that
  $$
  \mu\big(\Phi^{k_i}\big)=k_i\loopp(\Phi)+d_i, \textrm{ where }
  d_i\equiv \mu\big(\Phi^{\ell_i}\big)\mod 2N.
  $$
  In particular, the residue of $d_i$ in $\Z_{2N}$ is independent of
  $k_i$.
  
  Indeed, let us write $\Phi\in\tSp(2n)$ as the product $\phi\xi$,
  where $\phi$ is a loop and $\xi$ is a short path; see Section
  \ref{sec:loop}. Then $\loopp(\Phi)=\hmu(\phi)$ and
  $$
  \mu\big(\Phi^{k_i}\big)=k_i\hmu(\phi)+\mu\big(\xi^{k_i}\big).
  $$
  Let $\zeta$ be a geodesic in $\Gamma$ connecting the origin to
  $\Phi^{k_i}(1)$. We have
  $$
  \mu\big(\Phi^{k_i}\big)=k_i\hmu(\phi)+d_i, \textrm{ where }
  d_i=\mu(\zeta)+\hmu\big(\xi^{k_i}\zeta^{-1}\big).
  $$
  Here $\xi^{k_i}\zeta^{-1}$ is a loop in $\Gamma$, and hence
  $2N\mid \hmu\big(\xi^{k_i}\zeta^{-1}\big)$ since Condition
  \ref{cond:B1} is assumed to fail. Let $\tzeta$ be the geodesic close
  to $\zeta$, connecting the origin to $\Phi^{\ell_i}$. Such a
  geodesic exists once $\Phi^{k_i}(1)$ is close to $\Phi^{\ell_i}(1)$,
  and $\mu(\tzeta)=\mu(\zeta)$.  Then, it is not hard to see that
  $2N \mid \hmu \big(\Phi^{\ell_i}\tzeta^{-1} \big)$ from the
  condition that $2N\mid\ell_i$, and therefore
  $d_i\equiv \mu\big(\Phi^{\ell_i}\big)  \mod 2N$.

  Now we are in a position to modify the partition
  $\ell_1+\ldots+\ell_r=\ell$.  Namely, when
  $\mu\big(\Phi^{\ell_i}\big) \equiv n\mod 2N$, we take
  $k_i\in 2N\N+1$ such that $\Phi^{k_i}(1)$ is sufficiently close to
  $\Phi^{\ell_i}(1)$. Then $d_i\equiv n\mod 2N$ and
  $$
  \mu\big(\Phi^{k_i}\big)=k_i\loopp(\Phi)+d_i\equiv
  \loopp(\Phi)+n\not\equiv n\mod 2N.
  $$
  To show that $k_1+\ldots+k_r=:k$ is again an extremal partition one
  argues exactly as in the first part of the proof.
\end{proof}

\begin{proof}[Proof of Theorem \ref{thm:adm_part-dense}] Since
  $\Gamma=\T^n$, all eigenvalues of $\Phi(1)$ are necessarily distinct
  and in particular $\Phi(1)$ is semi-simple. Furthermore, the orbit
  $\Phi(1)^k$, $k\in\N$ is dense on $\T^n$. Hence for a suitable
  iterate $\Phi^\ell$, the end point $\Phi^\ell(1)$ is the sum of
  arbitrarily small rotations $\exp\big(\pi\sqrt{-1}\lambda_i\big)$,
  where $0<\lambda_i\ll|\lambda_n|$ for $i=1,\ldots, n-1$ and
  $\lambda_n<0$. Iterating again to bring
  $\exp\big(\pi\sqrt{-1}\lambda_nk\big)$ close to $1\in S^1$ while
  other components stay small due to the inequality between the
  eigenvalues, we can ensure that $\Phi^{m}(1)$ is a sum of
  arbitrarily small rotations, which we still denote by
  $\exp\big(\pi\sqrt{-1}\lambda_i\big)$ with all $\lambda_i>0$, and
  such that $\loopp(\Phi^m)=-2+d$ with $2N\mid d$.  We have
  $$
  \mu\big(\Phi^{m}\big)=n-2+d\equiv n-2\mod 2N.
  $$
  and, as long as $r\max|\lambda_i|<2$,
  $$
  \mu\big(\Phi^{rm}\big)=n+ r(d-2).
  $$
  Furthermore, $m+\ldots+m=rm=:k$ is an extremal iterations, for
  $$
  r \mu\big(\Phi^{m}\big)-(r-1)n=n+ r(d-2)=\mu\big(\Phi^{rm}\big).
  $$
\end{proof}

\begin{proof}[Proof of Theorem \ref{thm:adm_part-dim4}]
  As in the proof of Proposition \ref{prop:cond-A2}, we view
  $C=[0,1]^2 \subset \R^2$ as the fundamental domain for $\T^2$ and
  let $L_q$ stand for connected components of the inverse image $L$ of
  $\Gamma$ in $C$ under the map $C \rightarrow \T^2$. We first
  investigate when the Condition \ref{cond:A} fails. Since $L$ always
  intersects $\Pi_3=(0,1/3)^2 \subset C$ when the slope $s$ is
  positive, we assume that $s <0$ and write $s = - s_1 /s_2$ where
  $s_1, s_2 \in \N$ are relatively prime. If $\Gamma$ is connected,
  $L$ divides parallel boundary components of $C$ into $s_i$ equal
  segments (and more if $\Gamma$ is not connected). As a consequence,
  if $s_1 \geq 3$ or $s_2 \geq 3$ there exist $L_q \subset L$ such
  that $L_q \cap \Pi_3 \neq \emptyset$. Note that if $\Gamma$ is
  connected and $s= -1,\, -1/2,\, -2$, then $L \cap \Pi_3 = \emptyset$
  and Condition \ref{cond:A} fails by Corollary \ref{cor:cond-A2}
  (cf.\ Remark \ref{rmk:Cond-A-2D}). In other words, Condition
  \ref{cond:A} is satisfied for a connected group $\Gamma$ if and only
  if $s$ is not one of these values.

  In the remaining part of the proof we assume that Condition
  \ref{cond:A} is satisfied, i.e., $s \neq -1,\, -1/2,\, -2$, but
  Condition \ref{cond:B1} fails and Condition \ref{cond:B2} fails for
  the connected component of identity in $\Gamma$; cf. Remark
  \ref{rmk:connected}. In particular, for every loop $\gamma$ in
  $\Gamma$ we have $2N\mid\hmu(\gamma)$. Let $V$ be a convex
  neighborhood of $0 \in \T^2$ such that $V \cap \Gamma$ is connected
  and let $k \in \N$ be such that $\Phi^k(1) \in V\cap \Gamma$. By
  assumption, $2N \, \vert \, \loopp\big(\Phi^k\big)$. We replace
  $\Phi$ by $\Phi^k$ and $\Gamma(\Phi)$ by $\Gamma\big(\Phi^k\big)$,
  while keeping the notation $\Phi$ for the iterated map.

  Below we will show that for any $k\in\N$ the index
  $\mu\big(\Phi^k\big)$ only depends (modulo $2N$) on the connected
  component $L_q \subset L$ where the end point $\Phi^k(1)$ is, and,
  furthermore, for any consecutive connected components $L_{q_1}$,
  $L_{q_2}$ of $L$, the index is different modulo $2N$. Here $L_{q_1}$
  is consecutive to $L_{q_2}$ if $L_{q_1} \neq L_{q_2}$ and there is
  no other connected component of $L$ which is strictly closer to
  $L_{q_1}$ in $C$ than $L_{q_2}$.

  Indeed, let $L_q \subset L$ and $k_1, k_2 \in \N$ be such that
  $\Phi^{k_1}(1)$ and $\Phi^{k_2}(1)$ are in $L_q$. Since
  $2N \, \vert \, \loopp(\Phi)$, the difference
  $$
   d:= \mu\big(\Phi^{k_1}\big)-\mu\big(\Phi^{k_2}\big) \mod 2N
  $$
  is equal (modulo $2N$) to the mean index $\hmu(\gamma)$ of a loop
  $\gamma \subset \Gamma$. Since Condition \ref{cond:B1} fails, $2N$
  divides $ \hmu(\gamma)$ and hence $d \equiv 0 \mod 2N$. To prove the
  second assertion, suppose that $\Phi^{k_1}(1)$ and $\Phi^{k_2}(1)$
  are on two consecutive components $L_{q_1}$ and $L_{q_2}$. Let
  $\gamma$ be a path in $\Gamma$ connecting $\Phi^{k_2}(1)$ and
  $\Phi^{k_1}(1)$, and let $\eta$ be the shortest path in $C$ from
  $\Phi^{k_1}(1)$ to $\Phi^{k_2}(1)$. We have
  $$
  d \equiv \hmu(\gamma \sharp \eta) \mod 2N,
  $$
  where $\gamma \sharp \eta \subset \T^2$ is the loop obtained by
  concatenating $\gamma$ and $\eta$. If $d \equiv 0 \mod 2N$, by
  shifting the loop $\gamma \sharp \eta$, we see that the index is
  constant (modulo $2N$) as a function of $L_q$, which if $L$ is
  connected, i.e., $s=\pm 1$.

  With these observations in mind, we are ready to prove the theorem.
  Assume first that $s>0$ and either $s_1>3$ or $s_2 >3$. Then
  $L\cap \Pi_3$ has at least two connected components, and for at
  least on one of them the index is different from $2 \mod
  2N$. Combining this with the assumption that Condition \ref{cond:B1}
  fails finishes the proof for positive slopes: For $s>0$ the
  assertion of the theorem holds if $N=2$ and $s \neq 1,\, 3,\, 1/3$
  or $N=3$ and $s \neq 2,\, 1/2$. (When $s=1$ and $N=3$, Condition
  \ref{cond:B2} is satisfied.)

  When $s<0$ we will give different arguments for $N=2$ and $N=3$. In
  both cases there will be no other slope to rule out other than
  $s=-1,\, -2,\, -1/2$ (for which Condition \ref{cond:A} fails).

  \subsubsection*{The case of $N=2$} Since we are assuming that the
  Condition \ref{cond:B1} fails and thus $s_1 - s_2$ is even, both
  $s_1, s_2$ are odd. The index on the connected components of $L$
  that contain $0 \in \T^2$ is equal to $0 \mod 4$. Then since $s_1$
  and $s_2$ are odd, the index is $0 \mod 4$ on the connected
  component that is closest to $0 \in C$. We conclude that for $s<0$
  and $N=2$ the assertion holds if $s\neq -1, \, -2, \, -1/2$.

  \subsubsection*{The case of $N=3$} If either $s_1 \geq 6$ or
  $s_2 \geq 6$, then $L\cap \Pi_3$ has two connected components and
  the proof is finished as in $s>0$ case. It remains to check the
  following slopes: $s=-1/4,\, -4,\, -2/5,\, -5/2$. (Here we are again
  using the assumption that the Condition \ref{cond:B1} fails.)  A
  direct computation shows that when $s$ is from the list above, one
  of the connected components of $L$ that contains $0 \in \T^2$
  intersects $\Pi_3$. Furthermore, on such a component, as in the
  $N=2$ case, the index is equal to $0 \mod 6$. We again conclude that
  for $s<0$ and $N=3$ the assertion holds if $s\neq -1,\, -2,\, -1/2$.
\end{proof}

\subsection{Proof of Proposition \ref{prop:CZ-qm}}
\label{sec:adm_part-extra} Consider two elements $\Phi$ and $\Psi$ of
$\tSp(2n)$. Our goal is to prove the upper bound
$$
|D|=\big|\mu(\Psi\Phi)-\mu(\Psi)-\mu(\Psi)\big|\leq n
$$ 
on the absolute value of the defect $D$, where $\Phi$ and $\Psi$ and
the product $\Psi\Phi$ (or rather the end-points $\Phi(1)$ and
$\Psi(1)$ and their product) are non-degenerate. In fact, we only need
to show that $D\leq n$ for the opposite inequality $D\geq -n$ follows
by replacing $\Phi$ by $\Phi^{-1}$ and $\Psi$ by $\Psi^{-1}$.

It is convenient to recast the question in terms of the
Robbin--Salamon index $\MURS$; see \cite{RS-index}. For any path
$\Phi\colon [0,\,1]\to \Sp(2n)$ denote by $\gr(\Phi)$ the path traced
by the graph of $\Phi(t)$ in the Lagrangian Grassmannian of the
twisted product
$$
\bar{\R}^{2n} \times \R^{2n} =\big(\R^{2n} \times \R^{2n}, -\omega_0
\times \omega_0\big),
$$
where $\omega_0$ is the standard symplectic structure on
$\R^{2n}$. For a non-degenerate element $\Phi\in\tSp(2n)$, we have
$$
\mu(\Phi)= \MURS \big(\gr(\Phi), \triangle\big),
$$
where $\triangle$ is the diagonal in $\bar{\R}^{2n} \times \R^{2n}$.
Furthermore,
$$
D=\MURS \big(\gr(\Psi(1)\Phi), \triangle\big) -\MURS\big(\gr(\Phi),
\triangle\big),
$$
since the Robbin--Salamon index is additive under concatenation of
paths. The index is invariant under linear symplectic maps. Hence,
$$
\MURS \big(\gr(\Phi), \triangle\big) = \MURS\big(\gr(\Psi(1)\Phi),
\gr(\Psi(1))\big)
$$
and
$$
D=\MURS \big( \underbrace{\gr(\Psi(1)\Phi)}_{\Lambda},
\underbrace{\triangle}_{L_1}\big) - \MURS
\big(\underbrace{\gr(\Psi(1)\Phi)}_{\Lambda},
\underbrace{\gr(\Psi(1))}_{L_2}\big).
$$
In other words, in the notation introduced by the underbraces, we see
that $D$ can be expressed as the difference
$$
s\big(L_1, L_2; \Lambda(0), \Lambda(1)\big):=\MURS(\Lambda, L_1)-
\MURS(\Lambda, L_2),
$$
which is independent of the path $\Lambda$ connecting $\Lambda(0)$ to
$\Lambda(1)$ and called the H\" ormander index; see \cite[Thm.\
3.5]{RS-index}). Below we utilize this path independence to upper
bound the defect $D$.

Choose a Lagrangian complement $N$ to $\triangle$ which is transverse
to $\gr\big(\Psi(1)\big)$ and $\gr\big(\Psi(1)\Phi(1)\big)$. Every
Lagrangian subspace transverse to $N$ can be written as the graph of a
symmetric matrix with respect to the splitting $\triangle \times
N$. Indeed, observe that the graph of a linear map
$S\colon \triangle \rightarrow N$ is Lagrangian if and only if
$$
\omega_0(u+Su, v+Sv)=\langle u, Sv \rangle - \langle Su, v \rangle =0
$$
for every $u$, $v \in \triangle$, where
$\langle \cdot\,, \cdot \rangle$ is the standard inner product. This
is equivalent to the condition that $S$ is symmetric.

Let $A \colon \triangle \rightarrow N$ and
$B \colon \triangle \rightarrow N$ be such that
$\gr(A) = \gr\big(\Psi(1)\big)$ and
$\gr(B)=\gr\big(\Psi(1)\Phi(1)\big)$. Then
\begin{align*}
  D & = s\big(\triangle, \gr(\Psi(1)); \gr(\Psi(1)),
      \gr\big(\Psi(1)\Phi(1)\big)\big)\\
    & = s\big(\gr(0),\gr(A); \gr(A), \gr(B)\big).
\end{align*}
Next, applying \cite[Thm.\ 3.5]{RS-index} (the first equality) and
\cite[Lemma 5.2]{RS-index} (the second equality) to the right hand
side, we see that
\begin{align*}
D &=\frac{1}{2}\big[\sgn(B)-\sgn(A)-\sgn(B-A)\big]\\
& = \frac{1}{2}\sgn\big(B^{-1}-A^{-1}\big)\\
&  \leq n,
\end{align*}
since $\sgn(S)/2\leq n$ for any symmetric matrix $S$. This completes
the proof of the proposition.\hfill$\qed$

\begin{Remark}[Defect of the Conley--Zehnder Type Quasimorphisms] 
  \label{rmk:Maslov-qm}
  As an immediate consequence of Proposition \ref{prop:CZ-qm}, one
  obtains upper bounds on the defect $D$ of several types of Maslov or
  Conley--Zehnder quasimorphisms. Namely, it readily follows from the
  proposition that $|D|\leq 4n$ for the mean index and the upper or
  lower semi-continuous extensions of the Conley--Zehnder index. The
  proof of the proposition yields the upper bound $|D|\leq 3n$ for the
  Robbin--Salamon index.
\end{Remark}


\begin{thebibliography}{CKRTZ}

\bibitem[Ab]{Ab} A. Abbondandolo, \emph{Morse Theory for Hamiltonian
    Systems}, Chapman \& Hall/CRC Research Notes in Mathematics,
  425. Chapman \& Hall/CRC, Boca Raton, FL, 2001.

\bibitem[AS]{AS} A. Abbondandolo, M. Schwarz, Floer homology of
  cotangent bundles and the loop product, \emph{Geom.\ Topol.},
  \textbf{14} (2010), 1569--1722. 
 
\bibitem[AK]{AK} D.V.  Anosov, A.B. Katok, New examples in smooth
  ergodic theory. Ergodic diffeomorphisms, (in Russian), \emph{Trudy
    Moskov.\ Mat.\ Ob\v{s}\v{c}.}, \textbf{23} (1970), 3--36.

\bibitem[Ba]{Ba} A. Banyaga, Sur la structure du groupe des
  diff\'eomorphismes qui pr\'eservent une forme symplectique,
  \emph{Comment.\ Math.\ Helv.}, \textbf{53} (1978), 174–227.  

\bibitem[Br15a]{Br:Annals} B. Bramham, Periodic approximations of
  irrational pseudo-rotations using pseudoholomorphic curves,
  \emph{Ann.\ of Math.\ (2)}, \textbf{181} (2015), 1033--1086. 

\bibitem[Br15b]{Br} B. Bramham, Pseudo-rotations with sufficiently
  Liouvillean rotation number are $C^0$-rigid, \emph{Invent.\ Math.},
  \textbf{199} (2015), 561--580. 
 
\bibitem[BH]{BH} B. Bramham, H. Hofer, First steps towards a
  symplectic dynamics, \emph{Surv.\ Differ.\ Geom.}, \textbf{17}
  (2012), 127--178.

\bibitem[\c Ci]{Ci} E. \c Cineli, Conley conjecture and local Floer
  homology, \emph{Arch.\ Math.\ (Basel)}, \textbf{111} (2018),
  647--656.  

\bibitem[DeG$^2$P]{DG2P} M. De Gosson, S. De Gosson, P. Piccione, On a
  product formula for the Conley--Zehnder index of symplectic paths
  and its applications, \emph{Ann.\ Global Anal.\ Geom.}, \textbf{34}
  (2008), 167--183.
  
\bibitem[FK]{FK} B. Fayad, A. Katok, Constructions in elliptic
  dynamics, \emph{Ergodic Theory Dynam.\ Systems}, \textbf{24} (2004),
  1477--1520.

\bibitem[Gi]{Gi:CC} V.L. Ginzburg, The Conley conjecture, \emph{Ann.\
    of Math.\ (2)}, \textbf{172} (2010), 1127--1180.

\bibitem[GG15]{GG:survey} V.L. Ginzburg, B.Z. G\"urel, The Conley
  conjecture and beyond, \emph{Arnold Math.\ J.}, \textbf{1} (2015),
  299--337.

\bibitem[GG17]{GG:Rev} V.L. Ginzburg, B.Z. G\"urel, Conley conjecture
  revisited, \emph{Int.\ Math.\ Res.\ Not.\ IMRN}, \textbf{2017}, doi:
  10.1093/imrn/rnx137.

\bibitem[GG18a]{GG:PR} V.L. Ginzburg, B.Z. G\"urel, Hamiltonian
  pseudo-rotations of projective spaces, \emph{Invent.\ Math.},
  \textbf{214} (2018), 1081--1130.

\bibitem[GG18b]{GG:PRvR} V.L. Ginzburg, B.Z. G\"urel, Pseudo-rotations
  vs.\ rotations, Preprint ArXiv:1812.05782.

\bibitem[GGK]{GGK} V. Guillemin, V. Ginzburg, Y. Karshon,
  \emph{Cobordisms and Hamiltonian Group Actions}, Mathematical
  Surveys and Monographs, 98. American Mathematical Society,
  Providence, RI, 2002.

\bibitem[HS]{HS} H. Hofer, D. Salamon, Floer homology and Novikov
  rings, in \emph{The Floer Memorial Volume}, Progr.\ Math., vol.\
  133, Birkh\"auser, Basel, 1995, 483--524.
  
\bibitem[LMP95]{LMP1} F. Lalonde, D. McDuff, L. Polterovich, On the
  flux conjectures, in \emph{Geometry, Topology, and Dynamics
    (Montreal, PQ, 1995)}, 69--85, CRM Proc.\ Lecture Notes, 15,
  Amer.\ Math.\ Soc., Providence, RI, 1998.

\bibitem[LMP99]{LMP2} F. Lalonde, D. McDuff, L. Polterovich,
  Topological rigidity of Hamiltonian loops and quantum homology,
  \emph{Invent.\ Math.}, \textbf{135} (1999), 369--385. 

\bibitem[LeRS]{LeRS} F. Le Roux, S. Seyfaddini, work in progress.

\bibitem[Lo]{Lo} Y. Long, \emph{Index Theory for Symplectic Paths with
    Applications}, Birkh\"auser Verlag, Basel, 2002.

\bibitem[McD]{McD} D. McDuff, Hamiltonian $S^1$-manifolds are
  uniruled, \emph{Duke Math.\ J.}, \textbf{146} (2009), 449--507.
  
\bibitem[MS]{MS} D. McDuff, D. Salamon, \emph{J-holomorphic Curves and
    Symplectic Topology}, Colloquium publications, vol.\ 52, AMS,
  Providence, RI, 2004.

\bibitem[On]{On} K. Ono, Floer--Novikov cohomology and the flux
  conjecture, \emph{Geom.\ Funct.\ Anal.}, \textbf{16} (2006),
  981--1020. 
  
\bibitem[PSS]{PSS} S. Piunikhin, D. Salamon, M. Schwarz, Symplectic
  Floer--Donaldson theory and quantum cohomology, in \emph{Contact and
    Symplectic Geometry (Cambridge, 1994}), Publ.\ Newton Inst., vol.\
  8, Cambridge University Press, Cambridge, 1996, 171--200.

\bibitem[R\v S]{RS} D. Repov\v s, E.V. \v S\v cepin, A proof of the
  Hilbert-Smith conjecture for actions by Lipschitz maps, \emph{Math.\
    Ann.}, \textbf{308} (1997), 361–364.  

\bibitem[RS]{RS-index} J. Robbin, D. Salamon, The {M}aslov index for
  paths, \emph{Topology}, \textbf{32} (1993), 827--844.

\bibitem[Sa]{Sa} D.A. Salamon, Lectures on Floer homology, in
  \emph{Symplectic Geometry and Topology}, IAS/Park City Math.\ Ser.,
  vol.\ 7, Amer.\ Math.\ Soc., Providence, RI, 1999, 143--229.

\bibitem[SZ]{SZ} D. Salamon, E. Zehnder, Morse theory for periodic
  solutions of Hamiltonian systems and the Maslov index, \emph{Comm.\
    Pure Appl.\ Math.}, \textbf{45} (1992), 1303--1360.

\bibitem[Se08]{Se:biased} P. Seidel, A biased view of symplectic
  cohomology, in \emph{Current Developments in Mathematics},
  vol.\ 2006, Int.\ Press, Somerville, MA, 2008, 211--253.

\bibitem[Se15]{Se:prod} P. Seidel, The equivariant pair-of-pants
  product in fixed point Floer cohomology, \emph{Geom.\ Funct.\
    Anal.}, \textbf{25} (2015), 942--1007.

\bibitem[Sh19a]{Sh:HZ} E. Shelukhin, On the Hofer--Zehnder conjecture,
  Preprint ArXiv:1905.04769.
   
\bibitem[Sh19b]{Sh} E. Shelukhin, Pseudorotations and Steenrod
  squares, Preprint ArXiv:1905.05108.


\end{thebibliography}
\end{document}